\documentclass{article}
\usepackage{graphics}
\usepackage{graphicx}
 \usepackage[utf8]{inputenc}
 \usepackage[T1]{fontenc}
 \usepackage{lmodern}
 \usepackage[normalem]{ulem}
 \usepackage{verbatim}
 \usepackage{bbm}

 \usepackage{stmaryrd}

 \usepackage{amsmath}

 \usepackage{amssymb}
 \usepackage{dsfont}
\usepackage{amsthm}
\usepackage{pdfpages}
\usepackage{eso-pic}

\newtheorem{defi}{Definition}[section]
\newtheorem{theo}{Theorem}[section]%

\newtheorem{assum}{Assumption}[section]%
\newtheorem{cor}[theo]{Corollary}%
\newtheorem{lemma}[theo]{Lemma}%
\newtheorem{rem}[theo]{Remark}%
\newtheorem{prop}[theo]{Proposition}%

\def\mun{{\hat \mu_N}}
\def\ra{{\rightarrow}}
\newcommand{\arcosh}{arcosh}
\newcommand{\E}{\mathbb E}

\newcommand{\Pp}{\mathbb P}
\newcommand{\C}{\mathbb C}
\newcommand{\R}{\mathbb R}
\newcommand{\N}{\mathbb N}

\newcommand{\Ss}{\mathbb S}

\def\la{{\lambda}}

\begin{document}
\pagestyle{myheadings}
\markright{\hfill Large deviations for matrices with variance profiles\hfill}
\title{Large deviations for the largest eigenvalue of matrices with variance profiles}
\author{Jonathan Husson}

\maketitle

\begin{abstract}
In this article we consider Wigner matrices $(X_N)_{N \in \N}$ with variance profiles which are of the form $X_N(i,j) = \sigma(i/N,j/N) a_{i,j} / \sqrt{N}$ where $\sigma$ is a symmetric real positive function of $[0,1]^2$, either continuous or piecewise constant and where the $a_{i,j}$ are independent, centered of variance one above the diagonal. 
We prove a large deviation principle for the largest eigenvalue of those matrices under the condition that they have sharp sub-Gaussian tails and under some additional assumptions on $\sigma$. These sub-Gaussian bounds are verified for example for Gaussian variables, Rademacher variables or uniform variables on $[- \sqrt{3}, \sqrt{3}]$. This result is new even for Gaussian entries. 
\end{abstract}

\section{Introduction}
One of the key results in  random matrix theory is Wigner's theorem: it establishes the convergence of the empirical measure of the eigenvalues of Wigner matrices towards the semi-circular measure \cite{Wig58}. These Wigner matrices are a model of real or complex self-adjoint random matrices with independent centered subdiagonal entries of variance $1/N$ and independent centered diagonal entries of variance $O(1/N)$. Later Füredi and Koml\' os proved that the largest eigenvalue of such matrices converges almost surely toward $2$ \cite{FuKo} under an assumption of boundedness on the moments of the entries. This moment hypothesis was then relaxed to an hypothesis of boundedness for the fourth moment by Vu in \cite{vu05} which was later proved to be necessary by Lee and Yin in \cite{LY}. Similar results also exist for Wishart matrices (that is matrices of the form $\frac{1}{M} X^* X$ where $X$ is a $M \times N$ random matrix with i.i.d. centered entries of variance $1$) and for matrices with variance profiles (that is self-adjoint random matrices whose diagonal and subdiagonal entries are independent centered but whose variance of the entries may not be constant up to a factor $1/N$). In that case the limit of the empirical measure depends on the profile \cite{Girko2001}. 

Once one knows the limits of the empirical measure and the largest eigenvalue, one can wonder how the probability that they are away from these limits behaves. These questions are of great importance for instance in mobile communication systems \cite{bianchi2009performance,Fey_2008} and in the study of the energy landscape of disordered systems \cite{arous2017landscape,maillard2019landscape}. In the case of matrices from the Gaussian Orthogonal Ensemble or the Gaussian Unitary Ensemble, thanks to the the orthogonal or unitary invariance of the distributions, the joint law of the eigenvalues is explicitly known (see for example \cite{mehta}) and the spectrum behaves like a so-called $\beta$-ensemble. By Laplace's principle, once one takes care of the singularities, those formulas lead to large deviation principles both for the empirical measure \cite{BAG} and the largest eigenvalue \cite{BADG01}. 
 
In the case of general distributions, since eigenvalues are complicated functions of the entries, large deviations remain mysterious. Concentration of measure results were obtained in compactly supported and log-Sobolev settings by Guionnet and Zeitouni \cite{GZ}. Several recent breakthroughs proved large deviation principles for matrices with entries with distributions whose tails are heavier than Gaussian both for the empirical measure and the largest eigenvalue respectively by Bordenave and Caputo and by Augeri \cite{fanny,BordCap}. Those results rely on the fact that the large deviation behaviour comes from a small number of large entries. These ideas are further generalized to the questions of subgraphs counts and eigenvalues of random graphs in \cite{Fa18, CoDe, BhGa}. In the case of sub-Gaussian entries, a large deviation principle for the largest eigenvalue of matrices with Rademacher-distributed entries was proved by Guionnet and the author in \cite{GuHu18} using the asymptotics of Itzykson-Zuber integrals computed by Guionnet and Maïda in \cite{GM}. Indeed, one obtains the large deviations by tilting the measure by spherical integrals. Under this tilted law, the matrix is roughly distributed as a sum $W_N + R$ where $W_N$ is a Wigner matrix and $R$ a deterministic matrix of rank one. Then, the largest eigenvalue of such a deformed model is well known and follows the phenomenon of BBP transition (coined after Bai, Ben Arous and Péché who observed it in the case of deformations of sample covariance matrices \cite{BBP}). Notably the rate function for the large deviation principle of the largest eigenvalue of a matrix with Rademacher entries is the same as for the GOE. The crucial hypothesis verified by the Rademacher law that assure that the upper and lower large deviations bounds both coincides with those of the Gaussian case is the so-called sharp sub-Gaussianity. This property of the Rademacher law is expressed in terms of its Laplace transform: 

\[ \forall t \in \R, \int \exp(tx) d\mu(x) \leq \frac{\exp( \mu(x^2) t^2)}{2} \]

   For distributions that are sub-Gaussian but not sharply so, large deviation lower and upper bounds were also proved by Augeri, Guionnet and the author in \cite{AuGuHu} for large values and values near the bulk of the limit measure. In this case though our rate function near infinity can be strictly smaller to the rate function for the GOE. 

	Wigner's original approach to determine the limit of the empirical measure was to estimate the trace of moments of Wigner matrices but a more modern approach is to estimate the resolvent $(z - W_N)^{-1}$ using the Schur complement formula. One then finds that the Stieltjes transform $m$ of the limit measure must be a solution of the so called Dyson equation: 
\[ \frac{1}{m(z)} = z  -  m(z), \forall z \in \C \setminus \R \]
with the convention that the Stieltjes transform of $\mu$ is $ z \mapsto \mu( (z -x)^{-1})$. Furthermore, if for $z \in \mathbb{H}^+$ (where $\mathbb{H}^+$ is $\{ z \in \C: \Im(z) > 0 \}$) we set the condition $\Im m(z) < 0$, then the only solution to this equation for $z \in \mathbb{H}^+$ is $m(z) = (z - \sqrt{z^2 - 4})/2 $ which is the Stieltjes transform of the semicircular measure $\sqrt{(4 -x^2)_+} dx / 2 \pi$ . In the case of matrices with variance profiles, it can be computed again by the Schur complement formula applied on the resolvant $ G(z) = (z - W_N)^{-1}$, which shows that up to an error term, its diagonal terms satisfies the following equation which admits only one solution of negative imaginary part: 

\[ \frac{1}{G_{i,i}(z)} =  z - \sum_j s_{i,j} G_{j,j}(z), \forall z \in \mathbb{H}^+ \]
where $s_{i,j} = N \E[ |X_N(i,j)|^2]$.

Then, using that the Stieltjes transform of the empirical measure is $N^{-1} \sum_{i} G_{i,i}(z)$ one can find the limit measure. This equation has been used to study those matrices for instance by Girko in \cite{Girko2001} by Khorunzy and Pastur in \cite{KhoPa}, Anderson and Zeitouni in \cite{AndZei} and Schlyakhtenko in $\cite{Schlya96}$. It was extensively studied in itself by Alt, Erdös, Ajanki, Kreuger and Schröder in a series of articles where it is used to prove local laws and universality of the local eigenvalue statistics both on the bulk, the cusp and the edge of the spectrum \cite{Aja15,Cusp1,Cusp2,AEK2,AEK18,Aji17,ErdosEdge}. One may want to look at \cite{erdos2019matrix} for a more thorough review on the subject.

In this article, we will use the techniques developed in \cite{GuHu18} and apply them to random matrices with variance profiles to prove a large deviation principle for the largest eigenvalue. We will place ourselves in the same context of entries with sharp sub-Gaussian law. Such a result is new, even for matrices with Gaussian entries and once again our rate function will not depend on the laws of the entries. We will consider a symmetric (or Hermitian)  matrix model $(X_N)_{N \in \N}$ with independent sub-diagonal entries with a variance profile $\Sigma_N(i,j) =N^{1/2} \sqrt { \E[ |X_N(i,j)|^2]}$. We will consider a piecewise constant case where $\Sigma_N$ is equal to some $\sigma_{k,l}$ on squares of the form $I_k^{(N)} \times I_l^{(N)}$ for $k,l = 1,...,n$ where $(I_k^{(N)})_{1 \leq k \leq n}$ is a collection of disjoint intervals covering $\llbracket 1,N \rrbracket$ and such that $I_k^{(N)}/N$ converges to some non-trivial interval $I_k$ of $[0,1]$. In this case we will define $\sigma$ to be the piecewise constant function equal to $\sigma_{k,l}$ on $I_k \times I_l$. We will also consider the case of a variance profile which converges toward a continuous function $\sigma$ in the sense that $\lim_N \sup_{i,j} | \Sigma_N(i,j) - \sigma(i/N,j/N)| =0$. In both cases the empirical measure converges to a measure $\mu_{\sigma}$ characterized by the fact that its Stieltjes transform $m$ is equal to $m(z) = \int_0^1 \mathfrak{m}(x,z)dx$ where $\mathfrak{m}$ is the only solution of the equation (see \cite{Girko2001}): 

\[ \frac{1}{\mathfrak{m}(x,z)} = z - \int \sigma^2(y,x) \mathfrak{m}(y,z) dy, \forall z \in \mathbb{H}^+ \]

 with the condition that $\Im \mathfrak{m}(x,z) < 0$ for for all $x\in [0,1], z \in \mathbb{H}^+$.  
We will find then that large deviations of the largest eigenvalue occur when we tilt our measure so that $X_N$ has roughly the same law than  $\tilde{X}_N + R$ where $\tilde{X}_N$ is a random matrix with the same variance profile as $X_N$ and where $R$ is deterministic and of finite rank. Since $\tilde{X}_N$ will not be a Wigner matrix, finding the correct tilt will be more involved than in the classical case and will require some additional hypothesis on the variance profiles in order for the tilt to yield the desired lower bound.  

First, in section \ref{rate} we will introduce the rate function and the assumption on the variance profile we will need in order for our large deviation lower bound to coincide with our upper bound. In sections \ref{Scheme} to \ref{LDPlbBM} we will treat the case of matrices with piecewise constant variance profile which bears the most similarities with the models treated in \cite{GuHu18}. In these sections we will insist on the differences with \cite{GuHu18} while redirecting the reader to it for the parts of the proofs that stay the same. We will first prove a large deviations upper bound using an annealed spherical integral in section \ref{rfBM}. We will then tilt our initial measure to prove the lower bound in section \ref{LDPlbBM}. There we will use the assumption made in section \ref{rate} to prove we can find a good tilt. In section \ref{Approx} we will approximate the case of a continuous variance profile using piecewise constant ones. We will have to prove the convergence of the rate functions of the approximations. Since the approximations will only satisfy our lower bound up to an error term, we will also prove that this error can ultimately be neglected. In section \ref{2by2} we will illustrate the cases where our result applies in the simple context of a piecewise constant variance profile with four blocks. In the same section we will illustrate the limits of our approach and the necessity to make some assumptions concerning the variance profiles, with an example of a matrix whose variance profile does not satisfy our assumptions and such that the rate function for the large deviations of the largest eigenvalue does not match our rate function. Finally, in section \ref{ratefunc} we will discuss the explicit value of the rate function and in particular we will present a condition that when verified assures us that the rate function does depend on the variance profile only through the limit measure of the matrix model. 
\subsection*{Acknowledgement}
The author would like to thank Alice Guionnet for her help proofreading this article and particularly the introduction, Ion Nechita for bringing to his attention the example in Remark \ref{Ion} and the referees for their careful reading and sugeestions. This work was partially supported by the ERC Advanced Grant LDRaM (ID:884584).
\subsection{Variance profiles}\label{VarProf}
In the rest of the article, a real $x$ is said to be non-negative if $x \geq 0$, $\R^+$ is the set $\{ x \in \R: x\geq 0 \}$ and $\R^{+,*} = \R^+ \setminus \{0\}$. Our random matrix model will be of the form $W_N \odot \Sigma_N$, where $W_N$ is either a real or a complex Wigner matrix, $\Sigma_N$ is a real symmetric matrix and $\odot$ is the entrywise product. $\mathcal{P}(A)$, where $A$ is a measurable space, will denote the set of probability measures on $A$. We denote for $n\in\N$ and a set $A$, $\mathcal{S}_n(A)$ the set of symmetric matrices with entries in $A$. First of all, we describe the matrices $\Sigma_N$ we will be using. These matrices will converge as piecewise constant functions of $[0,1]^2$ to some function $\sigma$ on $[0,1]^2$ called the variance profile. We will consider here two cases: the case where $\sigma$ is piecewise constant and the case where it is continuous.

\textbf{Piecewise constant variance profile}:
We consider a variance profile piecewise constant on rectangular blocks. Let $n\in \N^*$, $\Sigma = ( \sigma_{i,j})_{i,j\in \llbracket 1,n \rrbracket }$ a real symmetric $n \times n$ matrix with non-negative coefficients and $\vec{\alpha} = (\alpha_1,...,\alpha_n) \in \R^n$ such that for every $i$, $\alpha_i >0$ and $\alpha_1 + ... + \alpha_n = 1$. In this context  we will consider $\Sigma_N$ defined by block by:

\[ \Sigma_N(i,j) = \sigma_{k,l} \text{ if } i \in I^{(N)}_k \text{ and } j \in I^{(N)}_l \]
where for all $N \in \N$, $\{ I_1^{(N)},..., I_n^{(N)} \}$ is a partition of $\llbracket 1, N \rrbracket$ such that for all $j \in \llbracket 1,n \rrbracket$, $I^{(N)}_j = \llbracket a^N_j +1, a^N_{j+1} \rrbracket $ where for every $N$, the $a_j^N$ are such that: 
\[ 0 = a_1^N \leq a_2^N ... \leq a_{n+1}^N = N \]  
and such that for $j \in \llbracket 1,n \rrbracket $:
\[ \lim_N \frac{|I^{(N)}_j|}{N} = \alpha_j \]
We then define $(\gamma_i)_{1\leq i\leq n}$ and $(I_i)_{1\leq i\leq n}$ by:
\[ \gamma_0 = 0 \text{ and } \forall j \neq 0,\gamma_j := \sum_{i=1}^j \alpha_i \text{ and } I_i = [ \gamma_{i-1}, \gamma_i[ .\]

We shall also denote $\sigma : [0,1]^2 \to \R^+$ the piecewise constant function  defined by 

\[ \sigma(x,y) = \sigma_{k,l} \text{ if } (x,y) \in I_k \times I_l .\]

This setting will be referred as the case of a piecewise constant variance profile associated to the parameters $\Sigma$ and $\vec{\alpha}$.

\textbf{Continuous variance profile}:
In this case, we will consider a real non-negative symmetric continuous function $\sigma : [0,1]^2 \to \R^+$ and for every $N$, we will consider a symmetric matrix with non negative entries $\Sigma_N$ such that the sequence $\Sigma_N$ satisfies:  

\[ \lim_{N \to \infty} \sup_{1 \leq i,j \leq N} \Big| \Sigma_N(i,j) - \sigma \left( \frac{i}{N}, \frac{j}{N} \right) \Big| = 0   .\]
In both cases, we will call $\sigma$ the variance profile of the matrix model.  
\subsection{The generalized Wigner matrix model}
For the Wigner matrix $W_N$, we will consider two cases, a real symmetric one when $\beta=1$ and a complex Hermitian one when $\beta=2$. For every $N$, we will consider a family of independent random variables $(a^{(\beta)}_{i,j})_{1 \leq i \leq j \leq N}$ such that $a^{(\beta)}_{i,j}$ has the distribution $\mu_{i,j}^N$. For $\beta =1$, $ a^{(\beta)}_{i,j}$ is a real random variable for all $i,j$ and for $\beta = 2$, $a^{(\beta)}_{i,j}$ will be a complex random variable for $i \neq j$ and a real random variable for $i=j$. These will be the unrenormalized entries of $W_N$.
We will assume that all the $\mu_{i,j}^N$ are centered: 

\[ \int x d \mu_{i,j}^N (x) = 0, \forall 1 \leq i \leq j \leq N .\] 

For $\beta = 1$ we assume that off-diagonal entries have variance $1$ and diagonal entries have variance $2$: 

\[\mu^N_{i,j}(x^{2})=\int_{\R} x^{2 } d\mu_{{i,j}}^N(x)=1, \forall 1\le i<j\le N,\qquad \mu^N_{i,i}(x^{2})=2, \quad \forall 1\le i\le N\, . \]

For $\beta = 2$, if we denote $x$ the function $z \mapsto \Re z $ and $y$ the function $z \mapsto \Im z $, we assume the following conditions on the variances of the entries:

\[ \mu_{i,j}^N (x^2) = \mu_{i,j}^N(y^2) = \frac{1}{2} , \, \mu^N_{i,j}(xy) = 0, \,  \forall 1 \leq 1 < j \leq N, \qquad \mu_{i,i}^N (x^2) = 1,\, \forall 1 \leq i \leq N . \]

If $\mu$ is a probability measure on some $\R^d$ with a covariance matrix $C$, we say that $\mu$ has a \emph{sharp sub-Gaussian Laplace transform} if 

\[ \forall u \in \R^d, T_{\mu}(u) := \int_{\R^d} \exp( \langle u, x \rangle) d \mu(x) \leq \exp \frac{ \langle u, C u \rangle}{2} .\]

For $\mu$ a measure on $\C$, we can identify $\R^2$ and $\C$ and then the Laplace transform $T_{\mu}$ can be expressed for $z \in \C$ as 
\[ T_{\mu}(z) = \int_{\C} \exp( \Re \bar{z} x ) d \mu(x) .\]

We will need to make the following assumption on the $\mu_{i,j}^N$: 

\begin{assum}\label{A0BM}
	For every $N \in \N$ and $1 \leq i \leq j \leq N$, the distribution  $\mu_{i,j}^N$ has sub-Gaussian Laplace transforms. 
	\end{assum}
In particular, we can notice that in the complex case, it implies that for $i \neq j$, $\forall z \in \C$: 
\[ T_{\mu_{i,j}^N}(z) \leq \exp \Big( \frac{|z|^2}{4} \Big) .\]
Examples of distributions that satisfy this sharp sub-Gaussian bound in $\R$ are the (centered) Gaussian laws the Rademacher laws $\frac{1}{2}(\delta_{-1} + \delta_1)$ and the uniform law on a centered interval. On $\C$, if $X$ is a random variable such that $\Re(X)$ and $\Im(X)$ are independent and have sharp sub-Gaussian Laplace transform, then $X$ has a sharp sub-Gaussian Laplace transform.   
\begin{rem}\label{remOfer}
From the sharp sub-Gaussian bound, we have the following bound on the moments of $\mu_{i,j}^{N}$ if Assumption \ref{A0BM} is verified for $\beta = 1$ and $X$ is a random variable of distribution $\mu_{i,j}^N$: 

\[ \E[X^{2k}] \leq (2 k)! (T_{\mu_{i,j}^N}(- 1)+ T_{\mu_{i,j}^N}(1))/2 \leq (2 k)! e^{\mu_{i,j}^N(x^{2})/2} \]
and 
\[ \E[|X|^{2k+1}] \leq \E[X^{2k+2}]^{2k+1/2k+2} \leq ((2 k+2)! e^{\mu_{i,j}^N(x^{2})/2})^{2k+1/2k+2}.\]
We have a bound of the form:
\[ \E[|X|^{k}] \leq C k! \]
for some universal constant $C$. From this bound, we have that for every $\delta>0$, there exists $\epsilon > 0$ that does not depend on the laws $\mu_{i,j}^N$ such that for $|t|\leq \epsilon$.
$$T_{\mu_{i,j}^N}(t)\ge  \exp\{\frac{(1-\delta) t^2 \mu_{i,j}^N(x^{2}) }{ 2} \}\,.$$
We have also that the $T_{\mu_{i,j}^N}$ are uniformly $C^3$ in a neighbourhood of the origin: for $\epsilon>0$ small enough $\sup_{|t|\le \epsilon}\sup_{i,j,N} |\partial_t^3\ln T_{\mu_{i,j}^N}(t)|$ is finite. In the complex case, with the same method we have a similar result, that is that for every $\delta >0$, there is $\epsilon > 0$ such that for every $z \in \C$ such that $|z|\leq \epsilon$: 
\[ T_{\mu_{i,j}^N}(z)\ge  \exp\{\frac{(1-\delta) |z|^2 }{4} \}\,. \]
for $i \neq j$. 
\end{rem}

For both those cases, we will need to use concentration inequalities to ensure that at the exponential scale we consider, the empirical measure of our matrices can be approximated by their typical value. To this we will need this classical assumption. 
\begin{assum}\label{ACBM}
There exists a compact set $K$ such that the support of all $\mu_{i,j}^N$ is included in $K$ for all $i,j\in\llbracket 1, N \rrbracket$ and all integer number $N$, or all $\mu_{i,j}^N$ satisfy a log-Sobolev inequality with the same constant $c$ independent of $N$. In the complex case, we will suppose also that for all $(i,j)$, if $Y$ is a random variable of law $\mu_{i,j}$, there is a complex $a \neq 0$ such that $\Re(aY)$ and $\Im(aY)$ are independent. 
\end{assum}

Now for $\beta=1$ or $2$ and $N \in \N$, given the family $(a^{(\beta)}_{i,j})$, we define the following Wigner matrices:

$$
W_N^{(\beta)}(i,j) ={\bigg\lbrace}
\begin{array}{l}
\frac{ a_{i,j}^{(\beta)}}{\sqrt{N}} \text{ when } i \leq j, \cr
\frac{ a_{j,i}^{(\beta)}}{\sqrt{N}} \text{ when } i > j\,.\cr
\end{array}$$

From these definitions we define $X_N^{(\beta)}$ a real (if $\beta =1$) or complex (if $\beta=2$) matrix with variance profile $\Sigma_N$ as:

\[ X_N^{(\beta)} := W_N^{(\beta)} \odot \Sigma_N \] 
where for two matrices $A= (a_{i,j})_{i,j\in \llbracket 1,n \rrbracket}, B= (b_{i,j})_{i,j\in \llbracket 1,n\rrbracket}$, $A \odot B$ is the matrix $(a_{i,j}b_{i,j})_{i,j \in \llbracket 1,n \rrbracket}$.

If $A$ is a self-adjoint $N \times N$ matrix, we denote  $\lambda_{\min}(A)=\lambda_1\le \lambda_2 \cdots\le \lambda_N=\lambda_{\rm max}(A)$ its eigenvalues and $\mu_A$ its empirical measure: 

\[ \mu_A = \frac{1}{N} \sum_{i=1}^N \delta_{\lambda_i} \]

For $X_N^{(\beta)}$, we will abbreviate $\mu_{X_N^{(\beta)}}$ in $\mun$ throughout the article. 
\subsection{Statement of the results}

First of all with this matrix model, we will state with the following theorem the existence of a limit in probability of the empirical measure $\mun$. This limit, which depends only on the limit $\sigma$ of the variance profile is described in more detail in the Appendix \ref{Emp} where this theorem is proved: 
\begin{theo}
Both in the piecewise constant and in the continuous case, the empirical measure $\mun$ converges weakly in probability toward a compactly supported measure $\mu_{\sigma}$ which only depends on $\sigma$.
\end{theo}

This theorem is in fact an almost direct consequence from \cite[Theorem 1.1]{Girko2001}. It can also be obtained in the piecewise case and in the continuous case with the additional assumption that $\sigma$ is $1/2$-Hölder by applying Lemma 9.2 from \cite{AltErKr} which itself uses stability results for the Dyson equation. Since here the Dyson equation of our setting is simpler, we present a more elementary proof of this result using rougher stability results in Appendix \ref{Emp}. We denote by $r_{\sigma}$ the rightmost point of the support of  $\mu_{\sigma}$. First of all, we have the following result for the convergence of the largest eigenvalue of $X_N^{(\beta)}$.

\begin{theo}\label{CVlambdamax}
Suppose that Assumption \ref{A0BM} holds. Both in the piecewise constant case and the continuous case, we have that $\lambda_{\max}(X_N^{(\beta)})$ converges almost surely toward $r_{\sigma}$. 
\end{theo}
This theorem is a generalization of the result of convergence of the largest eigenvalue toward $2$ in the Wigner case which was proved by Füredi and Koml\'os \cite{FuKo} for distributions with moments such that $\E[|a_{i,j}^{(\beta)}|^k ] \leq k^{Ck}$ for some $C > 0$ and then by Vu for distributions with finite fourth moment \cite{vu05}. For this result, we need only to have a bound of the form $\E[  |a_{i,j}^{(\beta)}|^k] \leq r_k$ for some sequence $(r_k)_{k \in \N}$ (this hypothesis is automatically verified with our sharp sub-Gaussian bound). There are numerous similar results of convergence for the largest eigenvalue in the literature for models similar to this one, unfortunately, to the author knowledge none seem to quite correspond to the level of generality we are going for in this paper (the most similar to our model would be Theorem 2.7 from \cite{AltLocation} but here we would like to allow for rectangular blocks). Therefore we will be using here a stronger kind of results, which are  the local law results from \cite{Aji17} (corollary 2.10) in the case of a positive piecewise variance profile. The non-negative case as well as the continuous case will be proven by approximation, the only technicality is to prove that when we approximate a variance profile $\sigma$ by a sequence of variance profiles $(\sigma_n)$, the rightmost point of the support of $\mu_{\sigma_n}$ converges toward the rightmost point of the support of $\mu_{\sigma}$ (see Lemma \ref{conv}). Although using the local law may seem excessive for the purpose of proving the convergence of the largest eigenvalue, its anisotropic version will end up being used in the large deviation lower bound in section \ref{LDPlbBM}.

For the following theorem, which states a large deviation principle for $\lambda_{\max}(X_N^{(\beta)})$, we will need Assumptions \ref{ArgMaxDis} and \ref{ArgMaxCont} respectively for the case of a continuous variance profile and for the case of a piecewise constant variance profile. These assumptions are more thoroughly discussed in section \ref{rate}. Assumption \ref{ArgMaxCont} states that the following optimization problem for $\psi \in \mathcal{P}([0,1])$: 

\[ \sup_{\psi \in \mathcal{P}([0,1])} \left\{ \frac{\theta^2}{\beta} \int_{[0,1]^2} \sigma^2(x,y) d\psi(x) d \psi(y) -\frac{\beta}{2} D( Leb ||\psi) \right\} \]
has a determination of its maximum argument that is continuous in $\theta$. 

Similarly, Assumption \ref{ArgMaxDis} states that the following optimization problem for $\psi \in (\R^+)^n$ such that $\sum \psi_i =1$: 

\[ \sup_{ \psi \in (\R^+)^n, \sum \psi_i =1} \left\{ \frac{\theta^2}{\beta} \sum_{i,j=1}^n \sigma^2_{i,j} \psi_i \psi_j + \frac{\beta}{2} \sum_{i=1}^n \alpha_i \left( \log\psi_i - \log\alpha_i \right) \right\} \]
has a determination of its maximum argument that is continuous in $\theta$. Both assumptions are necessary to obtain the large deviation lower bound.

\begin{theo} \label{maintheowBM} \label{maintheow2BM} Suppose Assumptions \ref{A0BM}, \ref{ACBM} hold. Furthermore suppose that Assumption \ref{ArgMaxDis} holds in the piecewise constant case or that Assumption \ref{ArgMaxCont} holds in the continuous case. Then, 
the law of the largest eigenvalue  $\lambda_{\rm max}(X_N^{(\beta)})$ of $X^{(\beta)}_N$ satisfies a large deviation principle with speed $N$ and good rate function  $I^{(\beta)}$ which is infinite on $(-\infty,r_{\sigma})$.

In other words, for any closed subset $F$ of $\mathbb R$,
$$\limsup_{N\rightarrow \infty }\frac{1}{N}\log \Pp \left(\lambda_{\rm max}(X_N^{(\beta)})\in F\right)\le -\inf_{F}I^{(\beta)}\,,$$
whereas for any open subset $O$ of $\mathbb R$,
$$\liminf_{N\rightarrow \infty }\frac{1}{N}\log \Pp \left(\lambda_{\rm max}(X_N^{(\beta)})\in O\right)\ge -\inf_{O}I^{(\beta)}\,.$$
The same result holds for the opposite of the smallest eigenvalue $-\lambda_{min}(X_{N}^{(1)})$. Furthermore $I^{(2)}= 2 I^{(1)}$

\end{theo}

The rate functions $I^{(\beta)}$ are defined in section \ref{rate}. Examples of variance profiles that satisfy our Assumptions \ref{ArgMaxDis} and \ref{ArgMaxCont} are also given in section \ref{rate}.






\section{The rate function}\label{rate}

We will now define the rate function $I^{(\beta)}$ in Theorem \ref{maintheowBM}. This is in fact done the same way as in \cite{GuHu18} with the supremum $\sup_{\theta \geq 0} ( J(\mu_{\sigma},\theta,x) - F(\theta) )$.

In this formula, $J(\mu_{\sigma},\theta,x)$ is the limit of $N^{-1}\log \E[\exp(N \theta \langle e, A_N e \rangle)]$ where $e$ is a unitary vector taken uniformly on the sphere and $A_N$ is a sequence of matrices such that the empirical measures converge weakly to $\mu_{\sigma}$ and such that the sequence of the largest eigenvalues of $A_N$ converges to $x$. $F(\theta)$ is the limit of $N^{-1} \log \E[\exp(N\theta \langle e, X_N e \rangle)]$ where the expectation is taken both in $X_N$ and $e$. We will first describe the quantity $F(\theta)$. 

\subsection{The asymptotics of the annealed spherical integral}

For $\sigma : [0,1]^2 \to \R^+$ a bounded measurable function and $\psi$ a probability measure on $[0,1]$, let us denote:

\[ P( \sigma, \psi) := \int_0^1 \int_0^1 \sigma^2(x,y) d \psi(x) d \psi(y) \]
and for $\theta > 0$:

\[ \Psi(\theta, \sigma,\psi) := \frac{\theta^2}{\beta} P( \sigma, \psi) - \frac{\beta}{2} D( Leb || \psi) \]
where $D(. || . )$ is the Kullback-Leibler divergence, that is for $\lambda,\mu \in \mathcal{P}([0,1])$: 

\[
D( \lambda||  \mu) = \begin{cases} \int_0^1 \log \left( \frac{d \lambda}{d \mu}(x) \right) d\lambda(x) \text{ if } \lambda \text{ is absolutely continuous with respect to } \mu \\
 + \infty \text{ if this is not the case} \end{cases}
\]
and $Leb$ is the Lebesgue measure on $[0,1]$.


We consider here the following optimization problem with parameter $\theta >0$ on the set $\mathcal{P}([0,1])$:

\begin{eqnarray} \label{EqMax}
 F(\sigma, \theta) := \sup_{\mu \in \mathcal{P}([0,1]) }\left\{ \frac{\theta^2}{\beta} P( \sigma,\mu ) - \frac{\beta}{2} D(\text{Leb}||\mu) \right\}.
\end{eqnarray}
First, let us study this problem with the following lemma: 

\begin{lemma}
If $\sigma$ is bounded and continuous, the supremum is achieved in \eqref{EqMax}. Furthermore, in both the continuous and the piecewise cases, the function $F$ is continuous in $\theta$. 
\end{lemma}

\begin{proof}
Let us take $\mu_n$ a sequence of measures such that $\frac{\theta^2}{\beta} P( \sigma, \mu_n ) - \frac{\beta}{2}D(\text{Leb}||\mu_n)$ converges toward $F(\sigma,\theta)$. By compactness of $\mathcal{P}([0,1])$ for the weak topology we can assume that this sequence converges weakly to some $\mu$. Since we assume $\sigma$ continuous, $P(\sigma,.)$ is continuous for the weak topology and so, $\lim_n P(\sigma,\mu_n) = P(\sigma,\mu)$. Furthermore, since $(\lambda,\mu) \mapsto D(\lambda||\mu)$ is lower semi-continuous, we have $\liminf_n D(\text{Leb}||\mu_n) \geq D(\text{Leb}||\mu)$ so that 
\[ \frac{\theta^2}{\beta} P(\sigma, \mu) - \frac{\beta}{2} D(\text{Leb}||\mu) \geq \limsup_n \left\{ \frac{\theta^2}{\beta} P(\sigma, \mu_n ) - \frac{\beta}{2} D(\text{Leb}||\mu_n)\right\} = F(\sigma,\theta)  .\]
Furthermore, we have for every $\mu \in \mathcal{P}([0,1])$, $|\Psi(\theta,\sigma,\mu) - \Psi(\theta',\sigma,\mu)| \leq ||\sigma^2||_{\infty} |\theta^2 - \theta'^2|/\beta$ and so $|F(\sigma,\theta) - F(\sigma,\theta')| \leq ||\sigma^2||_{\infty} |\theta^2 - \theta'^2|/\beta$.

\end{proof}
In section \ref{rfBM} we will prove that the following limit: 

\[ \lim_{N \to \infty} N^{-1} \log \E_{e,X^{(\beta)}_N}[ \exp( N \theta \langle e, X^{(\beta)}_N e \rangle)] =F(\sigma, \theta) \]
holds in the piecewise constant case. 

In the piecewise constant case, that is when $\sigma$ is defined with a matrix $(\sigma_{i,j})_{1 \leq i,j \leq n}$ and parameters $\vec{\alpha}$, the optimization problem that defines $F$ is a simpler one. Indeed, if we denote for $\vec{\psi}=(\psi_1,...,\psi_n) \in \R$:

\[ \vec{P}( \sigma, \psi_1,...,\psi_n) = \sum_{i,j=1}^n \sigma_{i,j}^2 \psi_i \psi_j \]
and 

\[ \vec{\Psi}(\theta,\sigma,\vec{\psi}) := \frac{\theta^2}{\beta} \vec{P}(\sigma, \psi_1,...,\psi_n) + \frac{\beta}{2} \left( \sum_{i=1}^n \alpha_i \log \psi_i - \sum_{i=1}^n \alpha_i \log \alpha_i \right) .\]
We have easily, replacing $\mu$ by $ \sum_{i=1}^n \alpha_i^{-1}\mu(I_i) Leb_{I_i}$ that 

\begin{eqnarray} \label{EqMaxDis}
 F( \sigma, \theta) = \max_{\psi_i \geq 0, \sum_1^n \psi_i =1 } \vec{\Psi}(\theta,\sigma,\vec{\psi})
\end{eqnarray}
where $Leb_{I_i}$ is the Lebesgue measure restricted to the interval $I_i$. 









\subsection{Definition of the rate functions}
Now, in order to introduce our rate functions we need first to introduce the function $J$. This function is linked to the asymptotics of the following spherical integrals:

$$I_N(X, \theta)=\E_e[ e^{\theta N\langle e, X e\rangle}]$$
where the expectation holds over $e$ which follows the uniform measure on the sphere $\mathbb S^{\beta N-1}$ of radius one (taken in $\R^N$ when $\beta=1$ and $\C^N$ when $\beta=2$).
Denoting $J_N$ the following quantity:  
$$J_N(X,\theta)=\frac{1}{N}\log I_N(X,\theta)$$ 
the following theorem was proved in \cite{GM}: 
\begin{theo}\cite[Theorem 6]{GM}\label{mylBM}

If $(E_N)_{N \in \N}$ is a sequence of $N \times N$ real symmetric matrices when $\beta=1$ and complex Hermitian matrices when $\beta=2$ such that:
\begin{itemize}
\item The sequence of the empirical measures $\mu_{E_N}$ of $E_N$ weakly converges to a compactly supported measure $\mu$,
\item There are two reals $\lambda_{\min}(E), \lambda_{\max}(E)$ such that $\lim_{N \to \infty} \lambda_{\min}(E_N) = \lambda_{\min}(E)$ and $\lim_{N \to \infty} \lambda_{\max}(E_N) = \lambda_{\max}(E)$,
\end{itemize}
and $\theta \geq 0$, then: 
\[ \lim_{N \to \infty} J_N(E_N,\theta) = J(\mu,\theta, \lambda_{\max}(E)) \]
\end{theo}
The limit $J$ is defined as follows.  For a compactly supported probability measure we define its Stieltjes transform $G_\mu$ by
\[ G_{\mu}(z) := \int_{\R} \frac{1}{z-t} d \mu(t) \]

We assume hereafter that $\mu$ is supported on a compact $[a,b]$. Then  $G_{\mu}$ is a bijection from $\R \setminus [a,b]$ to $] G_{\mu}(a), G_{\mu}(b) [ \setminus \{ 0 \}$ where $G_{\mu}(a), G_{\mu}(b)$ are taken as the limits of $G_{\mu}(t)$ when $t \to a ^{-}$ and $t \to b ^{+}$. We denote by $K_{\mu}$ its inverse and let $R_{\mu}(z) := K_{\mu}(z) - 1/z$ be its $R$-transform as defined by Voiculescu in \cite{FreeEntropyV} (both defined on $] G_{\mu}(a), G_{\mu}(b) [ $ and $G_{\mu}(a)$ and/or $G_{\mu}(b)$ if they are finite). Let us denote by $r(\mu)$ the right edge of the support of $\mu.$
$J$ is defined for any $\theta \ge 0,$ and $\lambda \ge r(\mu)$ by,
$$
J( \mu, \theta, \la):= \theta v(\theta, \mu, \la) -\frac{\beta}{2} \int \log\left(1+\frac{2}{\beta} \theta v(\theta, \mu, \la) - \frac{2}{\beta}\theta y\right) d\mu( y),
$$
with 
$$
 v(\theta, \mu, \la) := \left\{\begin{array}{ll}
                                R_\mu(\frac{2}{\beta} \theta), & \textrm{if } 0 \le \frac{2 \theta}{\beta} \le G_{\mu}(\lambda),\\
                                 \lambda - \frac{\beta }{2\theta}, & \textrm{if }\frac{2 \theta }{\beta}>  G_{\mu}(\lambda),
                                \end{array}
\right.
$$
In  both the piecewise constant and the continuous  case, we introduce our rate function $I^{(\beta)}$ as 

\[ I^{(\beta)}(\sigma,x) = - \infty \text{ for } x \in ]- \infty , r_{\sigma}[ \]
and 

\[ I^{(\beta)}(\sigma,x) =  \max_{\theta \geq 0} \left( J(\mu_{\sigma},\theta,x) - F(\sigma,\theta) \right) \]
where $\mu_{\sigma}$ is the limit measure of $X_N^{(\beta)}$, our Wigner matrix whose variance profile converges toward $\sigma$. 
\begin{lemma}
For $\beta =1,2$, $I^{(\beta)}(\sigma,.)$ is a good rate function. Furthermore $I^{(2)}(\sigma,.) = 2 I^{(1)}(\sigma,.)$
\end{lemma}

\begin{proof}
As a supremum of continuous functions, $I^{(\beta)}(\sigma,.)$ is lower semi-continuous. 
We want to prove that the level sets of $I^{(\beta)}(\sigma,.)$ that is the $\{ x \in \R : I^{(\beta)}(\sigma,x) \leq M \}$ are compact. It is sufficient to show that  $\lim_{x \to + \infty} I^{(\beta)}(\sigma, x ) = +\infty$. For any fixed $\theta >0$, we have $\lim_{x \to \infty}J(\mu_{\sigma},\theta,x) = + \infty$. And so since we have $I^{(\beta)}(\sigma,x) \geq J(\mu_{\sigma},\theta,x) - F(\sigma,\theta)$, $I^{(\beta)}$ is a good rate function. With the change of variables $\theta' = \theta/2$ in the case $\beta =2$, we have that $I^{(2)}(\sigma,.) = 2 I^{(1)}(\sigma,.)$.
\end{proof}

\subsection{Assumptions on the variance profile $\sigma$}
	In order to prove the large deviation lower bound in the piecewise constant case, we will need the following assumption on $\sigma$: 
\begin{assum} \label{ArgMaxDis}
There exists some continuous $\theta \mapsto (\psi_{i}^{\theta})_{i \in \llbracket 1,n \rrbracket}$ with values in $(\R^+)^n \cap \{ \psi : \psi_1+...+\psi_n =1 \}$ such that $\psi^{\theta}$ is a maximal argument of the equation \ref{EqMaxDis}, that is:
\[ \frac{\theta^2}{\beta} \vec{P}( \sigma, \psi_{1}^{\theta},...,\psi_{n}^{\theta}) + \frac{\beta}{2} \left( \sum_{i=1}^n \alpha_i \log \psi^{\theta}_i - \sum_{i=1}^n \alpha_i \log \alpha_i \right) = F(\sigma,\theta). \]
\end{assum}

As a more practical example, the following assumption implies Assumption \ref{ArgMaxDis}:
\begin{assum}\label{ConcaveDis}
The function $\psi \mapsto \langle \psi, \sigma^2 \psi \rangle$ is concave on the set of $\psi \in \R^n$ such that $\sum_{i=1}^{n}\psi_i = 1$. Equivalently, for all $\psi \in \R^n$ such that $\sum_{i=1}^n \psi_i = 0$, $\langle \psi, \sigma^2 \psi \rangle \leq 0$ (where $\sigma^2$ is the matrix $(\sigma_{i,j}^2)_{1 \leq i,j \leq n}$).
\end{assum}

\begin{rem}\label{Ion}
Examples of variance profiles that satisfies this assumption are the variance profiles  associated to some parameters $(\alpha_1,...,\alpha_n) \in (\R^{+,*})^n$ and $\sigma_{i,j} = \mathds{1}_{i \neq j}$ (where $\mathds{1}_{i \neq j}$ is the indicator function equal to $1$ if $i \neq j$ and $0$ if $i = j$). In the case $n=2$ this a linearization of a Wishart matrix as in \cite{GuHu18}.
\end{rem}


\begin{lemma}
Assumption \ref{ConcaveDis} implies Assumption \ref{ArgMaxDis}.
\end{lemma}

\begin{proof}
The function $\vec{\psi} \mapsto \frac{\theta^2}{\beta} \vec{P}(\sigma,\vec{\psi}) + \frac{\beta}{2}\sum_{i=1}^n \alpha_i \log \psi_i$ is strictly concave and since it tends to $- \infty$ on the boundary of the domain, it admits a unique maximal argument $\psi^{\theta}$ which is also the unique solution to the following critical point equation:

\[ f(\psi)= \frac{2 \theta^2}{\beta} \left( \sum_{j=1}^n \sigma_{i,j}^2 \psi_j \right)_{i=1,...,n} + \frac{\beta}{2}\left( \frac{\alpha_i}{\psi_i}\right)_{i=1,...,n} \in \text{Vect}(1,...,1) \]
where $\text{Vect}(1,...,1)$ is the subspace of $\R^n$ spanned by the vector whose coordinates are all $1$ .
We now want to apply the implicit function theorem to prove that $\theta \mapsto \psi^{\theta}$ is analytic. First of all, the equation above can be rewritten $\Pi f(\psi) =0$ where $\Pi$ is the orthogonal projection on $\text{Vect}(1,...,1)^{\bot}$.We have that for every $u \in \R^n$:

\[ \forall i =1, ...,n, (d f_{\psi}(u))_i = \frac{ 2 \theta^2}{\beta} (Su)_i - \frac{\beta u_i}{2\psi_i^2}. \]
where we denote $S = (\sigma^2_{i,j})_{1 \leq i,j \leq n}$. 

It suffices to show that $d( \Pi f)_{\psi}(u) = \Pi d f_{\psi} (u) \neq 0$ for $u \in Vect(1,...,1)^{\bot}$, that is $d f_{\psi} (u) \notin \text{Vect}(1,...,1)$. For such a $u$, we have 

\[ \langle u, d f_{\psi} (u) \rangle = \frac{ 2 \theta^2}{\beta} \langle u,Su \rangle - \frac{\beta}{2} \sum_{j=1}^n \frac{u_j^2}{\psi_j^2} \]

Since $u \in \text{Vect}(1,...,1)^{\bot}$ we have by Assumption \ref{ConcaveDis} $\langle u,Su \rangle \leq 0$ and therefore $\langle u, d f_{\psi} (u) \rangle < 0$. So $d f_{\psi} (u) \notin \text{Vect}(1,...,1)$ and we can apply the implicit function theorem. 
\end{proof}
  
Examples of variance profiles that satisfies Assumption \ref{ArgMaxDis} but not Assumption \ref{ConcaveDis} are provided in section \ref{2by2}. In the same section, we will also show that without any assumptions on $\sigma$, the method used in this article may fail as we can have a large deviation principle but with a rate function different from $I$. 

In the continuous case, we will need the following assumption:
 \begin{assum} \label{ArgMaxCont}
There exists some continuous $\theta \mapsto \psi^{\theta}$ (for the weak topology) from $\R^+$ to $\mathcal{P}([0,1])$ such that $\psi^{\theta}$ is a maximal argument of \ref{EqMax} that is: 

\[ F(\sigma,\theta) = \Psi(\theta,\sigma,\psi^{\theta}) \]
\end{assum}

As for the piecewise constant case, the following assumption implies \ref{ArgMaxCont}

\begin{assum} \label{ConcaveCont}
The function $P(\sigma,.)$ is concave on the set $\mathcal{P}([0,1])$ of probability measures on $[0,1]$. 
\end{assum}

\begin{lemma}
Assumption \ref{ConcaveCont} implies \ref{ArgMaxCont}. 
\end{lemma}


\begin{rem}
A family of $\sigma$ satisfying Assumption \ref{ConcaveCont} is given by $\sigma^2(x,y) = | f(x)- f(y) | + C$ where $f$ is an increasing continuous function and $C\in \R^+$. Indeed, if $f$ is an increasing and continuous function on $[0,1]$, there is a positive measure $\nu$ on $[0,1]$ such that $f(x) - f(0) = \int_0^x d\nu(t)$ and we have $\sigma^2(x,y) = C + |\int_x^y d \nu(t)| = \int_0^1 \tau_t(x,y) d \nu(t) + C$ where $\tau_t (x,y) = \mathds{1}_{ x \leq  t < y} + \mathds{1}_{ y \leq  t < x}$ and so 
\[ P( \sigma,\psi) = \int_0^1 P(\tau_t,\psi) d\nu(t) + C . \]
Since $P(\tau_t,\psi) = 2 \psi([0,t[) (1 - \psi([0,t[)$, $P(\tau_t,. )$ is concave and so is $P(\sigma,.)$.
\end{rem}

\section{Scheme of the proof}\label{Scheme}

The proof of Theorem \ref{maintheowBM} will follow a path similar to \cite{GuHu18} for the piecewise constant case and then for $\sigma$ continuous, we will approximate it by a sequence of piecewise constant profiles. In the piecewise constant case, we will insist on the differences with \cite{GuHu18} and novelties brought by the introduction of a variance profile and we will refer the reader to the relevant parts of \cite{GuHu18} for further details on the proofs that stay similar. 
First of all, we will prove that the sequence of distributions of the largest eigenvalue of $X_N^{(\beta)}$ is exponentially tight.

\subsection{Exponential tightness}\label{exptightsecBM}
We will prove the following lemma of exponential tightness:
\begin{lemma}\label{exptightBM} For $\beta=1,2$, assume that the distribution of the entries $a^{(\beta)}_{i,j}$ satisfy Assumption \ref{A0BM}. Then:

\[ \lim_{K \to + \infty} \limsup_{N \to \infty} \frac{1}{N} \log \Pp[ \lambda_{\max}(X_N^{(\beta)}) > K] = - \infty .\]
Similar results hold for $\lambda_{\min}(X_N^{(\beta)})$.
\end{lemma}
We will in fact prove a stronger and slightly more quantitative result that will also be useful when we will approximate continuous variance profiles using piecewise constant ones (we recall that $\odot$ is the entrywise product of matrices):  

\begin{lemma}\label{exptight2}
Let $\beta=1,2$ and let us assume that the distribution of the entries $a^{(\beta)}_{i,j}$ satisfy Assumption \ref{A0BM}. Let $\mathcal{A}_N$ be the following subset of symmetric matrices: 
\[  \mathcal{A}_N := \{ A \in \mathcal{S}_N(\R^+) | \forall i,j \in \llbracket 1, N \rrbracket, A(i,j) \leq 1 \} .\]
For every $M>0$ there exists $B >0$ such that: 
\[ \limsup_N \frac{1}{N} \sup_{A \in \mathcal{A}_N}\log \Pp[ || A \odot W^{(\beta)}_N|| \geq B]  \leq -M .\]

\end{lemma}

\begin{proof}
We will use a standard net argument that we recall here for the sake of completeness. Let us denote:

\[ Y_N^{(\beta)} := A \odot W^{(\beta)}_N .\]
Where  $A \in \mathcal{A}_N$. If $R_N$ is a $1/2$-net of $\Ss^{\beta N -1}$ for the classical Euclidian norm, using a classical argument (see the proof of Lemma 1.8 from \cite{GuHu18}), we have:


\begin{equation}\label{b2BM}\Pp[ ||Y_N^{(\beta)}|| \ge 4 K] \leq  9^{\beta N} \sup_{u,v \in R_{N}}\Pp[\langle Y_N^{(\beta)} u, v \rangle \ge  K  ]. \end{equation}

We next bound the probability of deviations of $ \langle Y_N^{(\beta)} v, u \rangle$ by using Tchebychev's inequality.
For $\theta\ge 0$ we indeed have

\begin{eqnarray}
\Pp[\langle Y_N^{(\beta)} u, v \rangle \ge   K  ] &\leq& \exp\{- N K\} \E[ \exp\{N \langle Y_N^{(\beta)} u, v \rangle\} ] \nonumber\\
& \leq & \exp\{- N K\} \E[ \exp \left\{
\sqrt{N}\left(2 \sum_{i<j} \Re (A(i,j) a_{i,j}^{(\beta)}u_i \bar v_j ) + \sum_i A(i,i) a_{i,i} u_i v_i \right)\right\}]\nonumber\\
& \leq & \exp\{- N  K\} \exp \left(\frac{\ N}{\beta} (2\sum_{i<j} |u_{i}|^{2} |v_{j}|^{2}+\sum_i  |u_i |^{2}|v_i|^2) \right)  \\ 
&\leq & \exp \left( N \left( \frac{1}{\beta } - K \right) \right) 
\label{qwBM}
\end{eqnarray}
where we used that the entries have a sharp sub-Gaussian Laplace transform and that $|A(i,i)| \leq 1$. This complete the proof of the Lemma with \eqref{b2BM}.
\end{proof}
With this result, we conclude that the sequence of the distributions of the largest eigenvalue of $X_N^{(\beta)}$ in Lemma \ref{exptightBM} is indeed  exponentially tight. Therefore it is enough to prove a weak large deviation principle. In the following we summarize the  assumptions on the distribution of the entries as follows:
\begin{assum} \label{assBM} 
Either the $\mu_{i,j}^{N}$ are uniformly compactly supported in the sense that there exists a compact set $K$ such that the support of all $\mu_{i,j}^N$ is included in $K$, or the  $\mu_{i,j}^N$ satisfy a uniform log-Sobolev inequality in the sense that there exists a constant $c$ independent of $N$ such that for all smooth function $f$:
$$\int f^{2}\log\frac{f^{2}}{\mu_{i,j}^{N}(f^{2})}d\mu^{N}_{i,j}\le c\mu^{N}_{i,j}(\|\nabla f\|_2^{2})\,.$$
 Additionally, the $\mu^{N}_{i,j}$ satisfy Assumption \ref{A0BM}.  
\end{assum}


\subsection{Large deviation upper and lower bounds}
To use the result of Lemma \ref{Convergence} in appendix \ref{Emp} which states convergence of the largest eigenvalue toward the edge of the support, as well as the isotropic local laws we will need the following positivity assumption (which is mainly technical and will be relaxed later by approximation): 
\begin{assum} \label{Pos}
In the piecewise constant case, $\forall i,j \in \llbracket 1,n \rrbracket , \sigma_{i,j} > 0$. 
\end{assum}

We shall first prove that we have a weak large deviation upper bound similar to theorem 1.9 in \cite{GuHu18}:
\begin{theo}\label{theowldubBM} Assume that we have a piecewise constant variance profile $\sigma$ and that Assumption \ref{assBM} holds. Let $\beta=1,2$. Then,  for any real number $x$,
$$\lim_{\delta\ra 0}\limsup_{N\ra\infty} \frac{1}{N}\log \Pp\left(\left|\lambda_{\rm max }(X_N^{(\beta)})-x\right|\le\delta\right)\le -I^{(\beta)}(\sigma,x).$$
\end{theo}
 The lower bound will however be slightly different since we need to take into account the error term $E$.  

\begin{theo}\label{lbBM}

Assume that we are in the case of a piecewise constant variance profile $\sigma$ and that assumptions \ref{assBM} and \ref{Pos} hold. Let $E : \R^+ \to \R^+$ be a non-negative function. Suppose that there exist continuous functions $\theta \mapsto ( \psi_i^{E,\theta})_{i \in [1,n]}$ such that: 

\[ \vec{\Psi}( \theta, \sigma, \psi^{E,\theta}) \geq F(\sigma,\theta) - E(\theta) \]

then, if we let $\tilde{I}(\sigma, x) := \sup_{\theta \geq 0} \Big[ J( \mu_{\sigma}, \theta,x) - F(\sigma,\theta) + E( \theta) \Big]$, we have for every $x \geq r_{\sigma}$ and any $\delta > 0$: 
\[ \liminf_{N\ra\infty} \frac{1}{N} \log \Pp[ |\lambda_{\rm max}(X_N^{(\beta)}) - x | \leq \delta] \geq -  \tilde{I}( \sigma, x). \]

\end{theo}

Then, we will show that when Assumption \ref{ArgMaxDis} is verified, we can take $E = 0$ and the main theorem follows. However, when we deal with the continuous case, we will approximate $\sigma$ by piecewise constant functions $\sigma^n$. But for $\sigma^n$ Assumption \ref{ArgMaxDis} will be verified only up to an error term $E^n$ that can be neglected when $n$ tends to infinity.


Proving that Lemma \ref{theowldubBM} is verified for $x < r_{\sigma}$ is done as in \cite[Corollary 1.12]{GuHu18} using the following Lemma and saying that a deviation of $\lambda_{\max}( X_N^{(\beta)})$ below $r_{\sigma}$ imply a deviation of $\mun$ (which cannot occur with probability larger than the exponential scale we are interested in).   
 
\begin{lemma}\label{convmunBM}  Assume  that the 
$\mu_{i,j}^N$ are uniformly compactly supported or satisfy a uniform log-Sobolev inequality. 
Then,  for $\beta=1,2$, there exists some sequence $\xi(N)$ converging to $0$ such that

\[ \limsup_{N \to \infty} \frac{1}{N} \log \Pp\left( d( \mun, \mu_{\sigma}) > \xi(N) \right) = - \infty \,.\]
with $d$ the Dudley distance

\end{lemma}
The sequence $\xi(N)$ in this lemma depends on the the quantities $|I_i^{(N)}|/N$ and on how fast they converge to $\alpha_i$.  
The proof of this lemma is in Appendix \ref{Emp}. 
The asymptotics of 
$$J_N(X,\theta)=\frac{1}{N}\log I_N(X,\theta)$$ 
are given by Theorem \ref{mylBM}. We will also need the following lemma, which is a result of continuity for the $J_N$ and where we denote by $\|A\|$ the operator norm of the matrix $A$ given by
$\|A\|=\sup_{\|u\|_2=1}\| Au\|_2$ where $\|u\|_2=\sqrt{\sum |u_i|^2}$.


\begin{lemma}\label{contBM}
	Given $\mu \in \mathcal{P}([0,1])$ a compactly supported probability measure, $\xi : \N \to \R^+$ an arbitrary sequence tending to $0$, for every $\theta >0$, every $M>0$, every $\rho \geq r(\mu)$ there exists a function $g : \R^+ \to \R^+$ going to 0 at 0 such that for any $\delta > 0$,  if we denote by  $\mathcal{B}_N$ the set of real symmetric or complex Hermitian matrices $B_N$ such that $d( \mu_{B_N}, \mu) < \xi(N)$, $ |\lambda_{\rm max}(B_N) - \rho | < \delta$, and $\sup_{N} ||B_N|| \le M $,   for $N$ large enough, we have: 

\[ \limsup_{N \to \infty} \sup_{B_N \in \mathcal{B}_N} |J_N(B_N, \theta) - J(\mu,\theta, \rho) | \leq  g(\delta). \]

\end{lemma}
\begin{proof}
	Given that $x \mapsto J(\mu,\theta,x)$ is continuous on $[r(\mu), + \infty[$ (for fixed $\theta$ and $\mu$), we can choose $g(\delta) > 0$ such that for $|\rho - \rho'| \leq \delta$, $|J(\mu,\theta,\rho) - J(\mu,\theta,\rho')| \leq g(\delta)/2$ and $\lim_{\delta \to 0}g(\delta) = 0 $. Then, if we assume by the absurd that the lemma is false, there exists a sequence of matrices $(A_N)_{N \in \N}$ such that $||A_N|| \leq M$, $d(\mu_{A_N}, \mu_{\sigma}) \leq \xi(N)$, $|\lambda_{\max}(A_N) - \rho | \leq  \delta$ and $| J_N(A_N,\theta) - J(\mu,\theta,\rho) | > g(\delta)$. But then by compactness, we can assume that up to extraction $\lim_{N \to \infty}\lambda_{\max}(A_N) = \rho' \in [ \rho - \delta, \rho+ \delta]$,$\lim_{N \to \infty}\lambda_{\min}(A_N) = \rho'' \geq - M $  and then applying Theorem \ref{mylBM}, we have that $\lim_{N \to \infty}J_N(A_N, \theta) = J(\mu_{\sigma}, \theta, \rho')$, implying $|J(\mu,\theta, \rho) - J(\mu,\theta, \rho') | \geq g(\delta)$ which is absurd. 
	\end{proof}  
Using Lemma \ref{exptightBM} and Lemma \ref{convmunBM}, and defining 
$$\mathcal A_{x,\delta}^{M}=\lbrace X: \left|\lambda_{\rm max}(X)-x\right|<\delta\rbrace\cap \{X: d(\mun, \mu_\sigma)< \xi(N)\} \cap \{ X:\|X\|\le M\}\,, $$
we have that for any $L>0$,  for $M$ large enough and for $N$ large enough
\[ \Pp\left[\left|\lambda_{\rm max}(X_N^{(\beta)})-x\right|<\delta \right]=\Pp\left[X_N^{(\beta)}\in \mathcal A_{x,\delta}^{M}\right]+ O(e^{-N L})\,. \]

Therefore it is enough to study the probability of deviations on the set where $J_N$ is continuous. The last item we need to this end is the asymptotics of the annealed version of the spherical integral defined by 

\[ F_N( \theta,\beta) = \frac{1}{N} \log \E_{X_{N}^{(\beta)}} \E_{e}[ \exp(N \theta \langle e, X_{N}^{(\beta)} e \rangle) ] \]
where $e$ is a unit vector independent of $X_N^{(\beta)}$ and where we take both the expectation $\E_{e}$ over $e$ and the expectation $\E_{X_{N}^{(\beta)}}$ over  $X_{N}^{(\beta)}$.
\begin{theo}\label{rftheoBM} Suppose Assumption \ref{assBM} holds and that $\sigma$ is a piecewise constant variance profile. 
	For $\beta=1,2$ and $\theta\ge 0$,
	$$ \lim_{N\ra\infty}F_N(\theta,\beta) = F(\sigma,\theta)$$
	where we recall that $F(\sigma, \theta)$ is defined in equation \eqref{EqMax}.  
	
\end{theo}
This counterpart to \cite[Theorem 1.17]{GuHu18} will be proven in section \ref{rfBM}.
We are now in position to get an upper bound for $\Pp\left[ X_N^{(\beta)}\in \mathcal A_{x,\delta}^{M}\right]$. 
In fact, using the result of Lemma \ref{contBM}, for any $\theta\ge 0$, 
\begin{eqnarray}
\Pp\left[ X_N^{(\beta)}\in \mathcal A_{x,\delta}^{M}\right]&=&\E\left[ \frac{I_N(X_N^{(\beta)},\theta)}{I_{N}(X_{N}^{(\beta)} ,\theta)}\mathds{1}_{ X_N^{(\beta)} \in 
 \mathcal A_{x,\delta}^{M}
}\right]\nonumber\\
&\le &\E[ I_N(X_N^{(\beta)},\theta)] \exp\{ - N\inf_{X\in \mathcal A_{x,\delta}^{M}} J_N(X,\theta)\}\nonumber\\
&\le&  \exp \{N(F(\sigma, \theta) -  J(\mu_{\sigma}, \theta, x) + g(\delta) + o(1))\} \label{jhBM}
\end{eqnarray}

(where $o(1)$ is some quantity converging to $0$ as $N \to + \infty$). Taking the $\log$, dividing by $N$ and optimizing in $\theta \geq 0$ then gives Lemma \ref{theowldubBM}. 

To prove the complementary lower bound, we shall prove the following limit:
\begin{lemma} \label{crucBM}
For $\beta=1,2$, with the assumptions and notations of Theorem \ref{lbBM}, for any $x>r_{\sigma}$ , there exists $\theta=\theta_x\ge 0$ such that for any $\delta>0$
and $N$ large enough,

$$\liminf_{N\ra\infty} \frac{1}{N} \log \frac{   \E[ \mathds{1}_{ X_N^{(\beta)}\in \mathcal A_{x,\delta}^{M}}
I_N(X_N^{(\beta)},\theta)]}{ \E[ I_N(X_N^{(\beta)},\theta)]}\ge -E(\theta_x)\,.$$
\end{lemma}
This lemma is proved by showing in section \ref{FinitePerturb} that the matrix whose law has been tilted by the spherical integral is approximately a finite rank perturbation of a matrix with the same variance profile, from which we can use the techniques developped to study the famous BBP transition \cite{BBP}. The conclusion follows since then
\begin{eqnarray*}
\Pp\left[ X_N^{(\beta)}\in \mathcal A_{x,\delta}^{M}\right]&\ge &
\frac{   \E[ \mathds{1}_{ X_N^\delta\in \mathcal A_{x,\delta}^{M}}
I_N(X_N^{(\beta)},\theta_{x})]}{ \E[ I_N(X_N^{(\beta)},\theta_x)]} \E[ I_N(X_N^{(\beta)},\theta_x)]
 \exp\{ - N\sup_{X\in \mathcal A_{x,\delta}^{M}} J_N(X,\theta_{x})\}  \nonumber\\
&\ge& \exp\{N(g(\delta)+F(\theta_{x},\beta) - E(\theta_x)- J(\mu_{\sigma}, \theta_{x}, x)+o(1))\} \\
&\ge& \exp\{-N( \tilde{I}_{\beta }(x)+ o(1))\} \\
\end{eqnarray*}
where we used Theorem \ref{rftheoBM} and Lemma \ref{crucBM}. The Theorem \ref{maintheowBM} follows in the case of piecewise constant variance profile satisfying Assumption \ref{Pos} by noticing that if Assumption \ref{ArgMaxDis} is verified then we can choose $E =0$. We will relax the Assumption \ref{Pos} by approximation in the same time we will treat the continuous case.

\section{Proof for the asymptotics of the annealed integral in Theorem \ref{rftheoBM}}\label{rfBM}

 In this section we determine that the limit of $F_{N}(\theta,\beta)$ is $F(\sigma,\theta)$ as in Theorem \ref{rftheoBM}. In fact, we prove the following refinement of this theorem  which shows that with our assumption
 of sharp sub-Gaussian tails, the vectors $e$ that make the dominant contributions are delocalized.
 \begin{prop}\label{tyuiBM}
Suppose Assumption \ref{A0BM} holds. For $j \in \llbracket 1, n \rrbracket$, let $\psi_j := \sum_{i\in I^{(N)}_j} |e_i|^2$  and  $V_{N}^{\epsilon}=\{ e \in \mathbb{S}^{\beta N-1} \, : \, \forall j \in \llbracket 1,n \rrbracket,\forall i \in I^{(N)}_j, | e_i |  \leq \sqrt{\psi_j} N^{-1/4 - \epsilon} \} $. Then, for 
$\epsilon\in (0,\frac{1}{4})$,
\begin{eqnarray*}
F(\sigma, \theta)&=&\lim_{N \to + \infty}  F_N( \theta,\beta)\\
& =&\lim_{N\ra\infty}\frac{1}{N} \log  \E_e[\mathds{1}_{e\in V_N^\epsilon}\E_{X_{N}^{(\beta)}} [\exp(N \theta \langle e, X_{N}^{(\beta)} e \rangle)] ] .\\
\end{eqnarray*}
\end{prop}

\begin{proof} There again, the proof is very similar to the proof of Theorem 1.17 in \cite{GuHu18}. 
By denoting $L_{\mu} = \log T_{\mu}$,  we have with $e \in \mathbb{S}^{\beta N -1}$ fixed and by expanding the scalar product $\langle e, X_N^{(\beta)} e \rangle$: 
\begin{eqnarray*}
	\E_{X_{N}^{(\beta)}}[ \exp(N \theta \langle e, X_N^{(\beta)} e \rangle) ] 
&= & 
\exp\{ \sum_{i<j} L_{\mu_{i,j}^{N}}(2 \Sigma_N(i,j) \theta \bar e_i e_j \sqrt{N}) + \sum_i L_{\mu_{i,i}^{N}}( \Sigma_N(i,i) \theta |e_i|^2 \sqrt{N} )\}  \\
 & \leq&   \exp\{\frac{2 N\theta^{2}}{\beta} \sum_{i<j} \Sigma_N(i,j)^2 |e_{i}|^{2}|e_{j}|^{2} 
+\frac{N\theta^{2}}{\beta} \sum_{i} \Sigma_N(i,i)^2 |e_{i}|^{4}\}
\end{eqnarray*}
where we used the independence of the $(a_{i,j}^{(\beta)})_{i\le j}$ and their sub-Gaussian character. Let us recall $\psi_j = \sum_{i \in I^{(N)}_j} |e_i|^2 $ and 
\[ \vec{P}(\sigma, \vec{\psi}) = \sum_{i,j =1}^n \sigma_{i,j}^2 \psi_i \psi_j.\]We deduce:
\begin{eqnarray*}
\E_{X_{N}^{(\beta)}}[ \exp(N \theta \langle e, X_N^{(\beta)} e \rangle) ] & \leq & \exp \left( N\frac{\theta^2}{\beta} \vec{P}(\sigma,\psi_1,...,\psi_n) \right) .
\end{eqnarray*}

But since $e$ is taken uniformly on the sphere, the vector $\psi = (\psi_1,...,\psi_n)$ follows a Dirichlet law of parameters $\left( \frac{\beta \alpha_1 N }{2},...,\frac{\beta \alpha_n N }{2} \right) + o(N)$. We have the following large deviation principle for this Dirichlet law: 
\begin{lemma}\label{LDPDir}
Let $n\in \N^*$, and $(\alpha_N)_{N \geq 0} = (\alpha_{1,N},...,\alpha_{n,N})_{N \geq 0} \in ((\R^{+,*})^n)^{\N}$ be a sequence of vectors such that $ \lim_{N \to \infty} \frac{\alpha_N}{N} = ( \alpha_1,...,\alpha_n)$ and $\alpha_i > 0$ for all $i \in \llbracket 1, n \rrbracket$. The sequence of Dirichlet laws $Dir_N = Dir(\alpha_{1,N},...,\alpha_{n,N})$ satisfies a large deviation principle with good rate function $I(x_1,...,x_n) = \sum_{i=1}^n \alpha_i (\log x_i - \log \alpha_i)$.
\end{lemma}

\begin{proof}
We denote $f_N$ and $f$ the functions defined on $D = \{ x \in (\R^{+,*})^n : x_1+...+x_n =1 \}$ by 

\[ f_N(x) = \sum_{i} \frac{\alpha_{i,N} -1}{N} \log x_i \]
\[ f(x) = \sum_{i} \alpha_i \log x_i .\]
For $x \in D$, let's denote $\tilde{x}= (x_1,...,x_{n-1})$ and $\tilde{D}$ the image of $D$ by this application (so that $\tilde{D} = \{ x \in (\R^{+,*})^{n-1} : x_1+...+x_{n-1} \leq 1 \}$).
We have $ Dir_N(d \tilde{x}) = Z^{-1}_N \exp(N f_N(x_1,..., x_{n-1},(1- x_1 -...-x_{n-1}))) d\tilde{x}$ where 
\[ Z_N= \int_{\tilde{D}} \exp(N f_N(x_1,..., x_{n-1},(1- x_1 -...-x_{n-1}))) dx_1 ... dx_{n-1} .\]

We have that on every compact of $\tilde{D}$, $f_N(\tilde{x}, 1 - \sum_{i=1}^{n-1} x_i)$ converges uniformly toward $f(\tilde{x}, 1 - \sum_{i=1}^{n-1} x_i)$ (which is continuous) and furthermore, for every $M > 0$ there is a compact $K$ of $\tilde{D}$ such that for $N$ large enough $f_N(\tilde{x}, 1 - \sum_{i=1}^{n-1} {x_i}) \leq - M$ for $x \notin K$. With both those remarks we deduce via a classical Laplace method that 

\[\lim_{N \to \infty} \frac{1}{N} \log Z_N = \max_{x \in D} f(x) = \sum_{i=1}^n \alpha_i \log \alpha_i .\]

Using again classical Laplace methods and the fact that $x \mapsto \tilde{x}$ is a homeomorphism between $D$ and $\tilde{D}$, we have that the uniform convergence of $f_N$ and the continuity of the limit $f$ gives a weak large deviation principle with rate function $f(x) - \sum_{i=1}^n \alpha_i \log \alpha_i$ and the bound outside compacts gives the exponential tightness. The large deviation principle is proved.
\end{proof}

Using Lemma \ref{LDPDir} and Varadhan's lemma, we have since $\vec{P}$ is continuous that: 

\begin{multline*}
\lim_{N \to \infty} \frac{1}{N} \log \E_{e} \Big[ \exp \left( N \frac{\theta^2}{\beta} \vec{P}(\sigma,\psi_1,...,\psi_n) \right) \Big] = \\
\sup_{\psi_1,...,\psi_n \in [0,1], \psi_1+...\psi_n =1}\Big\{ \frac{2\theta^{2}}{\beta} \vec{P}(\sigma,\psi_1,...,\psi_n) - \sum_{i=1}^n \frac{\beta \alpha_i }{2}\log(\psi_i)- \sum_{i=1}^n \frac{\beta \alpha_i }{2}\log(\alpha_i)\Big\}
= F(\sigma, \theta)
\end{multline*}
so that we have proved the following upper bound: 
\begin{equation}\label{ubBM}
 \limsup_{N\ra\infty} F_N( \theta,\beta) \le
 F(\sigma,\theta). \end{equation}
 For the lower bound, we then again follow \cite{GuHu18} and we use that if  $e \in V_N^{\epsilon}$ then: 
 \[ \lim_{N \to + \infty} \sup_{i,j \in \llbracket 1,N \rrbracket} \sup_{e \in V^{\epsilon}_N} |2 \sqrt{N} \Sigma_N(i,j) \theta e_i e_j| = 0 .\]
 We can then use the Taylor expansions of $L_{\mu_{i,j}}$ near $0$ to prove that for any $\delta >0$:
 
 \begin{align}
 \E_e[\E_{X_N^{(\beta)}} [ \exp (  N \theta \langle e, X_N^{\beta}  e \rangle ) ]  ]&\geq  
 \E_e[\mathds{1}_{e \in V^{\epsilon}_N} e^{N \frac{\theta^2}{\beta} \vec{P} (\sigma,\psi_1,...,\psi_n)   (1- \delta)} ]\,. \label{mn0BM}
 \end{align}  
%

We shall then use the following lemma:
\begin{lemma} \label{vecBM}For any  $\epsilon\in (0,1/4)$ we have
\[ \lim_{N \to \infty} \Pp_{e}[ e \in V^{\epsilon}_N ] =1 \]\,
\end{lemma}
and that the event $\{ e \in V^{\epsilon}_N \}$ is independent of the vector $(\psi_1,...,\psi_n)$. As a consequence, we deduce from \eqref{mn0BM} that for any $\delta>0$ and $N$ large enough
\[ \liminf_{N\to \infty} F_N( \theta,\beta) \geq \liminf_{N \to \infty} \frac{1}{N} \log \E_e \left[ \exp \left((1 - \delta) N \frac{\theta^2}{\beta} \vec{P}(\sigma,\psi_1,...,\psi_n)\right) \right]. \]

and then we let $\delta$ tends to $0$. In order to see that the event $\{ e \in V_N^{\epsilon} \}$ is independent of $\psi$ and to prove Lemma \ref{vecBM}, we say that if we denote $e^{(j)} = (e_i)_{i \in I^{(N)}_j}$, $f^{(j)} := e^{(j)}/ ||e^{(j)}||$ is a uniform unit vector on the sphere of dimension $\beta|I^{(N)}_j| -1$. Furthermore all these  $f^{(j)}$ together with the random vector $(\psi_1,...,\psi_n)$ form a family of independent variables. Indeed, if we construct $e$ as a renormalized standard Gaussian variable, that is $e = g / ||g||$ then one can see that $ f^{(j)} = g^{(j)}/ ||g^{(j)}||$ and $\psi_j = ||g^{(j)}||^2/||g||^2$ where $g^{(j)}$ is defined from $g$ the same way as $e^{(j)}$ is defined from $e$. The independence of the $f^{(j)}$ and $\psi$ then comes from a classical change of variables. We notice then that in term of the $f^{(j)}$, we have:
\[ \{ e \in V_N^{\epsilon} \} = \bigcap_{j=1}^n \{ \forall i \in I_j^{(N)}, |f_i^{(j)}| \leq N^{-1/4 - \epsilon} \}. \]

 Then using the independence of the $f^{(j)}$:  

\[ \Pp_e[ e \in V^{\epsilon}_N] = \prod_{j=1}^{n} \Pp[ \forall i \in I^{(N)}_j, |f^{(j)}_i| \leq N^{-1/4 - \epsilon} ] .\]

The result follows since each of these terms converges to $1$.

\end{proof}

\section{Large deviation lower bound} \label{LDPlbBM}

We will now prove Theorem \ref{lbBM}. For a  vector $e$ of the sphere $\mathbb S^{\beta N-1}$ and $X$ a random symmetric or Hermitian matrix, we denote by $\Pp^{(e,\theta)}_N$ the tilted probability measure defined by: 

\[ d \Pp_N^{(e,\theta)}(X) = \frac{ \exp( N \theta \langle X  e,e \rangle) }{ \E_{X_N^{(\beta)}} [\exp( N \theta \langle X_N^{(\beta)} e,e \rangle)]} d \Pp_N(X)\]
where  $\Pp_N$ is the law of $X_{N}^{(\beta)}$.
Let us show that we only need to prove the following lemma: 

\begin{lemma}\label{prelim}Let us assume that the hypotheses of Theorem \ref{lbBM} hold.
Let $(\delta_N)_{N \in \N}$ be some arbitrary sequence of positive real numbers converging to $0$, $W_N$ the subset of the sphere $\mathbb S ^{\beta N -1}$defined by: 
\[ W_N := \{ e \in \mathbb S ^{\beta N -1} : \forall i,  \left| ||e^{(i)}||^2 - \psi_i^{E, \theta} \right| \leq \delta_N \} \]
where $E$ and $\psi^{E,\theta}$ are as in the hypotheses of Theorem \ref{lbBM} and $e^{(i)}=(e_j)_{j \in I_i^{(N)}}$ is the $i$-th block of entries of $e$.
For any $x \geq r_{\sigma}$, there is $\theta_x$ such that: 

\[ \lim_{N \to \infty} \inf_{e \in V_N^{\epsilon} \cap W_N} \Pp^{(e, \theta_x)}[ X_N^{(\beta)}\in \mathcal A_{x,\delta}^{M}] =1 \]

where we recall that 

$$\mathcal A_{x,\delta}^{M}=\lbrace X: \left|\lambda_{\rm max}(X)-x\right|<\delta\rbrace\cap \{ X : d(\mun, \mu_{\sigma})< \xi(N)\} \cap \{ X: \|X\|\le M\}\,. $$

\end{lemma}

\begin{proof}[Proof that Lemma \ref{prelim} implies Theorem \ref{lbBM}]
In the rest of the proof, we will abbreviate $\E_{X_N^{(\beta)}}$ in $\E_X$. We only need to prove that if there exists $E, \psi_i^{E,\theta}$ that satisfy the hypotheses of Theorem \ref{lbBM}, for every $x \geq r_{\sigma}$, there exists $\theta_x \geq 0$ such that:

$$\liminf_{N\ra\infty} \frac{1}{N} \log \frac{   \E_X[ \mathds{1}_{ X_N^{(\beta)}\in \mathcal A_{x,\delta}^{M}}
I_N(X_N^{(\beta)},\theta_x)]}{ \E_X[ I_N(X_N^{(\beta)},\theta_x)]}\ge - E(\theta_x)\,.$$

We have
\begin{eqnarray*}
\E_X[ \mathds{1}_{ X_N^{(\beta)}\in \mathcal A_{x,\delta}^{M}}
I_N(X_N^{(\beta)},\theta_x)]&=&\E_{e}[ \Pp_{N}^{(e,\theta_x)}[X_N^{(\beta)}\in \mathcal A_{x,\delta}^{M}] \E_X [\exp( N \theta_x \langle X_N^{(\beta)}e,e \rangle)]] \label{tyuBM}.\end{eqnarray*}
For $\delta >0$, let: 

\[ W_N^{\delta} := \{ e \in \mathbb S^{\beta N-1} : \forall i, \left||| e^{(i)} ||^2 - \psi_i^{\theta_x,E}\right|\leq \delta \} .\]

We have, using Lemma \ref{LDPDir} that: 

\[ \lim_{\delta \to 0} \liminf_{N \to \infty} \frac{1}{N} \log \Pp[ e \in W_N^{\delta}]  =  \frac{2}{\beta}\sum_{i=1}^n \alpha_i (\log \psi_i^{\theta,E} - \log \alpha_i).
\]

Let $(\delta_N)_{N\in \N}$ be a sequence converging to $0$ such that:

 \[ \liminf_{N \to \infty} \frac{1}{N} \log \Pp[ e \in W_N^{\delta_N}] \geq \frac{2}{\beta}\sum_{i=1}^n \alpha_i (\log\psi_i^{\theta_x,E} - \log\alpha_i) \]

and let: 
\[  W_N := \{ e \in \mathbb S^{\beta N-1} : \forall i, \left||| e^{(i)} ||^2 - \psi_i^{\theta_x,E}\right|\leq \delta_N \}. \]

We have then, using the Taylor expansions of the $L_{\mu_{i,j}^N}$ and that $\delta_N \to 0$ as in equation \eqref{mn0BM}, the following limit uniformly in $e\in W_N\cap V_N^{\epsilon}$:

\[ \lim_{N \to \infty} \sup_{e \in W_N\cap V_N^{\epsilon}}\Big| \frac{1}{N} \log \E_X [\exp( N \theta_x \langle X_N^{(\beta)}e,e \rangle)]]  -  \theta_x^2 \vec{P}(\sigma,\psi^{E, \theta_x})\Big|= 0.\]

Then we have:

\begin{eqnarray*}
\E_{e}[ \Pp_{N}^{(e,\theta_x)}[X_N^{(\beta)}\in \mathcal A_{x,\delta}^{M}] \E_X [\exp( N \theta_x \langle X_N^{(\beta)}e,e \rangle)]] &\ge & 
\E_{e}[ \mathds{1}_{e \in V_N^{\epsilon} \cap W_N} \Pp_{N}^{(e,\theta_x)}[X_N^{(\beta)}\in \mathcal A_{x,\delta}^{M}] \E_X [\exp( N \theta_x \langle X_N^{(\beta)} e,e \rangle)]] \\ &\ge & \E_{e}[ \mathds{1}_{e \in V_N^{\epsilon} \cap W_N} \Pp_{N}^{(e,\theta_x)}[X_N^{(\beta)}\in \mathcal A_{x,\delta}^{M}] e^{-N \theta_x ^2 \vec{P}(\sigma,\psi^{E, \theta_x})+ o(N)}]
\end{eqnarray*}

so we have that: 
\begin{align*}
\frac{   \E_X[ \mathds{1}_{ X_N^{(\beta)}\in \mathcal A_{x,\delta}^{M}}
I_N(X_N^{(\beta)},\theta_x)]}{ \E_X[ I_N(X_N^{(\beta)},\theta_x)]} &\geq \E_{e}[ \mathds{1}_{e \in V_N^{\epsilon} \cap W_N}\Pp^{(e,\theta_x)}[X_N^{(\beta)}\in \mathcal A_{x,\delta}^{M}] ]e^{-N (\theta_x^2 \vec{P}(\sigma,\psi^{E, \theta_x}) - F(\theta_x) + o(1))} \\
& \geq \Pp_e[e \in V_N^{\epsilon} \cap W_N] \inf_{f \in V_N^{\epsilon} \cap W_N} \Pp^{(f,\theta_x)}[X_N^{(\beta)}\in \mathcal A_{x,\delta}^{M}] e^{-N (\theta_x^2 \vec{P}(\sigma, \psi^{E, \theta_x}) - F(\theta_x) + o(1))} \\
& \geq \Pp_e[ e \in V_N^{\epsilon}] \Pp_e[ e \in W_N] \inf_{f \in V_N^{\epsilon} \cap W_N} \Pp^{(f,\theta_x)}[X_N^{(\beta)}\in \mathcal A_{x,\delta}^{M}] e^{-N (\theta_x^2 \vec{P}(\sigma, \psi^{E, \theta}) - F(\theta) + o(1))} \\
& \geq e^{-N(E(\theta_x)+ o(1))}
\end{align*}
where we used that $\{ e \in V_N^{\epsilon} \}$ and $\{e \in W_N\}$ are independent (since $\{ e \in V_N^{\epsilon} \}$ only depends on the $f^{(j)}$ and $W_N$ only depends on the $\psi_i$) and that $\frac{1}{N} \log \inf_{e \in V_N^{\epsilon} \cap W_N} \Pp^{(e,\theta_x)}[X_N^{(\beta)}\in \mathcal A_{x,\delta}^{M}]$ converges to $0$. So we have our lower bound. 
\end{proof}

And so it remains to prove the Lemma \ref{prelim}.  
More precisely, we will show that  for $\epsilon \in (\frac{1}{8},\frac{1}{4})$, for any $x>r_{\sigma}$ (the rightmost point of the support of $\mu_{\sigma}$) and $\delta>0$ we can find $\theta_x\ge 0$ so that for $M$ large enough,
\begin{equation}\label{lbwBM} \lim_{N \to \infty} \inf_{e \in V^\epsilon_N \cap W_N}   \Pp_{N}^{(e,\theta_{x})}[X_N^{(\beta)} \in \mathcal A_{x,\delta}^{M}]= 1 \,.\end{equation}
To prove \eqref{lbwBM}, we have to show that uniformly on $e \in V_N^{\epsilon} \cap W_N$ we still have that $ \lim_{N \to \infty}\Pp_N^{(e), \theta}[ d( \mun, \mu_{\sigma}) \geq \xi(N) ] = 0$ and that for $K$ large enough $\lim_{N \to \infty} \Pp^{(e,\theta)}_N[ ||X_N^{(\beta)}|| \ge K ]= 0$. This is done as in Lemma 5.1 in \cite{GuHu18}. 
Hence, the main point of the proof is to show that:
\begin{lemma}\label{difBM} Pick $\epsilon\in ]\frac{1}{8},\frac{1}{4}[$. For any $x>r_{\sigma}$, there exists $\theta_x$ such that 
for every $\eta >0$, 
\[ \lim_{N \to \infty} \sup_{e \in V^\epsilon_N \cap W_N} \Pp^{(e,\theta_x)}_N[ |\lambda_{\rm max} - x| \geq \eta ] = 0. \]
\end{lemma}

\begin{proof}\label{FinitePerturb}

For $e \in V^\epsilon_N$ fixed, let $X^{(e),N}$ be a matrix with  law $\Pp^{(e,\theta)}_N$. We will prove that $X^{(e),N}$ can be written as an additive perturbation of a random self-adjoint matrix $\tilde{X}^{(e),N}$ with independent sub-diagonal entries with the same variance profile as $X_N^{(\beta)}$: 

\[ X^{(e),N} = \tilde{X}^{(e),N} + ESE^* + \Delta^{(e),N} + Y^{(e),N}. \]

Here $ \Delta^{(e),N} = \E[ X^{(e),N}] - ESE^* $ and $ Y^{(e),N} = ( X^{(e),N} -\E[ X^{(e),N}])  \odot  ( I_N -  V^{(e),N} ) $ where for all $1 \leq i,j \leq N$:
\[V^{(e),N}_{i,j} = \frac{(1 + \mathds{1}_{i=j, \beta=1}) \Sigma_N(i,j)}{ \sqrt{N Var( X^{(e),N}_{i,j})}} . \]
The deterministic matrix $S$ is defined by $S := ( \sigma^2_{i,j})_{1 \leq i,j \leq n} $ and $E$ is the $ N \times n$ matrix defined by:
\[ E = \begin{pmatrix} & 0 & \cdots & 0 \\ e^{(1)} & \vdots &  & \vdots \\
& 0 & \cdots & 0 \\

0 &  & \cdots & 0 \\
\vdots & e^{(2)} & \hdots & \vdots \\
0 &  & \cdots & 0 \\
\vdots & \vdots & \ddots & \vdots \\
0 & 0 & \cdots & \\
\vdots & \vdots & & e^{(n)} \\
0 & 0 & \hdots & 
\end{pmatrix} \]
where we recall the the $e^{(i)}$ are the vectors defined by $ e^{(i)} = (e_j)_{j \in I_{j}^{(N)}}$. In particular, one can notice that the entries of $ESE^*$ are given by $(ESE^*)_{i,j} = \Sigma^2_N(i,j) e_i \bar{e}_j$. 
Furthermore the terms $\Delta^{(e),N}$ and $ Y^{(e),N}$ are negligible in operator norm for large $N$ in the sense that:

\begin{itemize}
\item There is a constant $C$ such that $|| \Delta^{(e),N}|| \leq C N^{- 1/2 - 4 \epsilon}$. 
\item For every $\delta>0$: 
\[ \lim_{N \to + \infty} \sup_{e \in V^\epsilon_N} \Pp_N^{(e),N}[ || Y^{(e), N} || > \delta] =  0. \]
\end{itemize}
Those estimates revolve around the Taylor expansion of the $L_{\mu_{i,j}}$ and the uniform bound on their derivatives near $0$ given by Remark \ref{remOfer}. Here we will only expose the computation justifying that the entries of $ \Delta^{(e),N}$ and $Y^{(e),N}$ tend to $0$. For how to refine this estimates and obtain that $ \Delta^{(e),N}$ and $Y^{(e),N}$ are negligible in operator norm, we refer the reader to 
the subsection 5.1 of \cite{GuHu18}.

%
%
%


We can express the density of $\Pp_N^{(e,\theta)}$ as the following product: 

\[ \frac{ d \Pp_N^{(e,\theta)} }{d \Pp_{N}}(X) = \prod_{i \le  j} \exp( 2^{1_{i\neq j}} \theta \Sigma_N(i,j) \sqrt{N} \Re(e_i \bar{e_j} a^{(\beta)}_{i,j}) - L_{\mu^N_{i,j}}( 2^{1_{i\neq j}} \theta \Sigma_N(i,j) \sqrt{N} e_i \bar e_j ))\]
where  the $a^{(\beta)}_{i,j}$ are defined as in the introduction. Since the $a^{(\beta)}_{i,j}$ independent (for $i \leq j$), the entries $X^{(e),N}_{i,j}$ remain so and their mean is given by the derivative of $L_{\mu_{i,j}}$:

\begin{eqnarray*} (\E[ X^{(e),N}])_{i,j} &=& \frac{ \Sigma_N (i,j) L_{\mu^N_{i,j}}'( 2 \sqrt{N} \Sigma_N(i,j) \theta e_i \bar e_j )}{\sqrt{N}} \\ &=& \frac{2 \theta }{\beta} \Sigma^2_N(i,j) e_i \bar e_j + \delta(2 \Sigma_N(i,j) \sqrt{N} \theta e_i \bar e_j)\frac{ \Sigma^3_N(i,j)  N \theta^2 |e_i|^2 |e_j|^2}{ \sqrt{N}} 
\end{eqnarray*}
if $i \neq j$, and  if $i=j$

\begin{eqnarray*} (\E[ X^{(e),N} ])_{i,i} &=& \frac{ \Sigma_N(i,i) L_{\mu_{i,i}^N}'( \sqrt{N} \Sigma_N(i,i) \theta |e_i|^2 )}{\sqrt{N}} \\ &=&
\frac{2 \theta}{\beta} \Sigma^2_N(i,i)e_i \bar e_i+\delta( \sqrt{N} \Sigma_N(i,i) \theta |e_i|^2) \frac{\Sigma^3_N(i,i)  N \theta^2 |e_i|^4}{ \sqrt{N}} 
\end{eqnarray*}
where we used that by centering and variance one,  $L_{\mu^N_{i,j}}'(0) = 0$, $ Hess  L_{\mu^N_{i,j}}(0) = \frac{1}{\beta} Id$ for all $i\neq j,N$,  $L''_{\mu^N_{i,i}}(0)=\frac{2}{\beta}$ for all $i,N$,  and where 
$$ |\delta(t)| \leq  4 \sup_{|u| < |t|}\max_{i,j,N} \{|L_{\mu^N_{i,j}}^{(3)}(u)|\}\,.$$
In the complex case, the notation $|L^{(3)}|$ just means 
\[ |L^{(3)}(u)| = \max_{|z| = 1}  \Big| \frac{d^{3}}{dt^3} L(u + tz)_{t=0} \Big| .\]
Hence, we have
$$ \Delta^{(e),N}_{i,j}={ \Sigma^3_N(i,j) \delta(2 \sqrt{N} \Sigma_N(i,j) \theta e_i \bar e_j) \sqrt{N} \theta^2 |e_i|^2 |e_j|^2}, 1\le i,j\le N\,.$$

Furthermore, when we identify $\C$ to $\R^2$ when $X^{(e),N}_{i,j}$ is a complex variable the covariance matrix of $X^{(e),N}_{i,j}$ is given by the Hessian of $L_{\mu_{i,j}}$ so that the variance of $X^{(e),N}_{i,j}$ is given by the Laplacian of $L_{\mu_{i,j}}$ (i.e. $\partial_z\partial_{\bar z}L_{\mu_{i,j}}$):

\[ Var(X^{(e),N}_{i,j}) = \frac{\Sigma_N(i,j)^2 }{N} \partial_z\partial_{\bar z} L_{\mu_{i,j}^N}( \theta \Sigma_N(i,j) e_i \bar e_j \sqrt{N})  = \Big( \frac{2}{\beta} \Big) ^{\mathds{1}_{i = j} }\frac{  \Sigma_N(i,j)^2 }{N}\Big( 1  + \sqrt{N} \Sigma_N(i,j) |e_ie_j| \delta( \theta \sqrt{N} \Sigma_N(i,j)  |e_ie_j|)  \Big)\]

And so : 

\[ (I_N - V^{(e),N})_{i,j} = 1 - \frac{1}{\sqrt{ 1 + \sqrt{N} \Sigma_N(i,j) |e_ie_j| \delta( \theta \sqrt{N} \Sigma_N(i,j)  |e_ie_j|)}} \]
So that $\lim_{N \to \infty} \max_{i,j} (I_N - V^{(e),N})_{i,j} = 0$. 

%
%

 And so, to conclude we need only to identify the limit of $\lambda_{\rm max}(\widetilde{X}^{(e),N} + \frac{2\theta}{\beta} E S E^{*})$. The eigenvalues of  
$\widetilde{X}^{ (e),N} + \frac{2\theta}{\beta} ES E ^{*}$ satisfy the following equation in $z$
$$0=\det (z- \widetilde{X}^{(e),N} - \frac{2\theta}{\beta}  E S E^{*})=\det (z- \widetilde{X}^{(e),N} )\det (1- \frac{2\theta}{\beta} (z- \widetilde{X}^{ (e),N} )^{-1}ES E^{*})$$ 
and therefore $z$ is an eigenvalue away from the  spectrum  of $ \widetilde{X}^{(e),N} $ if and only if

\[ \det (1- \frac{2\theta}{\beta} (z- \widetilde{X}^{ (e),N} )^{-1}ES E^{*}) = 0. \]
Recall that if $\mathbb{K}$ is a field and $A,B$ are two matrices  respectively in $ \mathcal{M}_{n,p}(\mathbb{K})$ and $ \mathcal{M}_{p,n}(\mathbb{K})$ then we have $\det(I_n + AB) = \det(I_p + BA)$. Using this, we have that the preceding equality is equivalent to

\[ \det(I_n - \theta' E^* (z - \widetilde{X}^{ (e),N})^{-1} E S  ) = 0\]
 
 where $\theta' = 2 \theta/ \beta$. 
\begin{lemma}
For $i,j \in \llbracket 1,n \rrbracket$, $\eta > 0$, $a > r_{\sigma}$, we have:

\[ \sup_{ e\in V^{\epsilon}_N\cap W_N} \Pp\left[\sup_{z \geq a} \left|(E^* (z - \widetilde{X}^{ (e),N})^{-1}  E S)_{i,j}-\sigma^2_{i,j} \psi_j^{\theta}m_j(z) \right|  \geq \eta \right] \to 0. \]
Where $m$ is the solution of the canonical equation $K_{\sigma}$ and $m_i$ the value taken by $m$ on the interval $I_i$ (see Appendix A for the definition of $K_{\sigma}$ and $m$). 
\end{lemma}

\begin{proof}
Let $G_N(z) := (z - \widetilde{X}^{ (e),N})^{-1}$, $M_N(z) = diag(m_{1,N}(z),....,m_{N,N}(z))$, where $m^N_x := \sum_{i=1}^{N} \mathds{1}_{Nx \in ]{i-1},i]}m_{i,N}$ and $m^N$ is the solution of $K_{\sigma_N}$  (the canonical equation defined in Appendix \ref{Emp} with $\sigma_N$ being the approximation of $\sigma$ defined in Theorem \ref{Girko}).
If we denote $\tilde{e}^k = ( \mathds{1}_{j \in I^{(N)}_k} e_j)_{j=1,...,N}$ the vector $e$ where we replaced all entries by $0$ except for the $k$-th block. 
\[ (E^* (z - \widetilde{X}^{ (e),N})^{-1} E S )_{i,j} = \sum_{k = 1}^n (\tilde{e}^i)^* G_N(z) \sigma^2_{k,j} \tilde{e}^k \]
So since $\lim_{N \to \infty} \sup_{e \in V^{\epsilon}_N\cap W_N}  \left| ||\tilde{e}^k||^2 - \psi_k^{\theta} \right| =0$ we only need to prove that for $k,l\in \llbracket 1,n \rrbracket$: 

\[ \lim_{N \to \infty} \sup_{e\in \mathbb{S}^{\beta N-1}} \sup_{z \geq a} \Pp[ |(\tilde{e}^{k})^{*} G_N(z) \tilde{e}^l -  \delta_{k,l} m_k(z) \psi_k^{\theta}| \geq \eta ] =0 .\]

To that end, we want to apply the anisotropic local law from \cite{Aji17} but in order to do so, we need to check its assumptions. (A) is verified since the variance profiles are uniformly bounded. (B) is verified with the Assumption \ref{Pos}. (D) is verified with the sub-Gaussian bound. To verify (C), we apply \cite[Theorem 6.1]{Aja15}.
Thanks to \cite[Theorem 1.13]{Aji17}, if we fix some $\gamma >0 $, $D >0$, $\epsilon >0$, for N large enough: 

\[ \sup_{e,f \in \mathbb{S}^{\beta N-1}} \sup_{z \in \C, \Im z \geq N^{\gamma -1}} \Pp[ |e^{*} G_N(z) f -  e^{*} M_N(z)f| \geq N^{-1/10} ] \leq N^{-D} .\]

Furthermore following Theorem \ref{Convergence}, we have that for $a' \in  ] r_{\sigma},a[ $, $D>0$, N large enough 

\[ \Pp[ \lambda_{\max}(\widetilde{X}^{ (e),N}) \geq a' ] \leq N^{-D} .\]

Let $e,f \in \mathbb{S}^{\beta N -1}$ and $h:z \mapsto e^{*} G_N(z) f$ and $k:z \mapsto e^{*} M_N(z) f$.
On the event $\{ \lambda_{\max}(\widetilde{X}^{ (e),N}) < a' \}$, we have that $|h(z)|,|k(z)|\leq \frac{1}{(\Re z - a')}$ and $|h'(z)|,|k'(z)| \leq \frac{1}{(\Re z -a')^2}$ for $\{ z :\Re z > a' \}$ and therefore, for $\gamma < 1/10$, we can in fact assume that our bound holds for any $z$ such that $\Re(z) > a$ and in particular for $z$ real (up to some multiplicative constant $C$ before the $N^{-1/10}$). Let 
\[ A_N := \{ a + k/N : k \in [0, N^2] \}.\]

By union bound, we have that for $N$ large enough:

\[ \Pp[ \sup_{z \in A_N }|(e^{*} G_N(z) f -  e^{*} M_N(z)f)| \geq C N^{-1/10} ] \leq N^{-D +2} .\]

Combining this again with the bounds of the derivative of $h$ and $k$ and the bound in modulus that is derived from the bound on $ \lambda_{\max}(\widetilde{X}^{ (e),N})$, we get for some $C' > 0$: 

\[ \Pp[ \sup_{z > a} |e^{*} G_N(z) f -  e^{*} M_N(z)f| \geq C'  N^{-1/10} ] \leq N^{-D +2} \]
 for $N$ large enough. Furthermore, this bound is uniform in $e$ and $f$. We then use Theorem \ref{Stab2} and the bounds one the derivatives of $M_N(z)$ the same way to conclude that for any $\eta > 0$,  for $N$ large enough and $e \in V_N^{\epsilon} \cap W_N$ we have that (for $N$ large enough): 

\[ \Pp[ \sup_{z > a} | (\tilde{e}^{k})^{*} (G_N(z))\tilde{e}^{l} - \delta_{k,l} \psi_{k}^{\theta} m_k(z)| \geq \eta |] \leq N^{-D+2} \]
where $m$ is the solution of $K_{\sigma}$ and $m_i$ is the value taken by $m$ on the interval $I_i$. And so we have: 

\[ \Pp[ \sup_{z > a}| (E^* G_N(z) SE)_{i,j} - m_j(z) \sigma^2_{i,j} \psi_j^{\theta} | > \eta ] \leq N^{-D+2}\text{ for every } i,j \in \llbracket 1,n \rrbracket. \]

\end{proof}
Let's denote  $D(\theta,z)$ the diagonal $n \times n$ matrix $\text{diag}(m_1(z)\psi_1^{\theta},..., m_n(z)\psi_n^{\theta})$, we have that the above limit can be rewritten $S D( \theta,z)$. From the preceding lemma we have that for $\eta > 0$  uniformly in $e \in \mathbb{S}^{\beta N -1}$ that 
\[ \Pp[ \sup_{z > a}| \det(I_n - \theta' E^* (z - \widetilde{X}^{ (e),N})^{-1} E S  ) - \det( I_n - \theta' S D(\theta,z))| \geq \eta] \leq N^{-D} \]
for $N$ large enough. 

So since $\lim_{z \to \infty} \det( I_n - \theta' S D(\theta,z)) =1$, all that remains is to solve the determinantal equation: 
\[ \det( I_n - \theta ' S D(\theta,z)) =0 \]
and the largest solution $z > r_{\sigma}$, if it exists,  will be the the limit of $\lambda_{\max}$.  We can rewrite this equation: 
\begin{equation} \label{Deter}
\det(I_n - \theta' \sqrt{D(\theta,z)}S\sqrt{D(\theta,z)}) =0.
\end{equation}

Let $\rho(\theta,z)$ be the largest eigenvalue of $\sqrt{D(\theta,z)}S\sqrt{D(\theta,z)}$. Then, the largest $z$ solution of equation (\ref{Deter}) is the unique solution of: 

\begin{equation}\label{rho}
 \theta' \rho(\theta,z) = 1
\end{equation}
 
one $]r_{\sigma}, + \infty[$. Indeed, with $\theta$ fixed, if $\theta' \rho(\theta,z) = 1$ then $z$ is a solution of \eqref{Deter}. Since the $z \mapsto m_i(z) $ are strictly decreasing, so is $\rho(\theta,.)$. So for $z' > z$, $\theta '\rho(\theta,z') < 1 $ and so $z'$ cannot be solution of \eqref{Deter} for the same $\theta$. Similarly, if $z$ is a solution of (\ref{Deter}) then $\theta' \rho(\theta,z) \geq 1$. If $\theta' \rho(\theta,z) >  1$ then since $z \mapsto \theta' \rho(\theta,z)$ is continuous and decreasing toward 0, there exists $z' >z$ such that $\theta '\rho(\theta,z') = 1$ and $z'$ is therefore a solution of (\ref{Deter})  strictly larger than $z$. 

Therefore, it suffices to prove that for any $x > r_{\sigma}$ there is at least one $\theta_x$ such that 
\[ \theta_x' \rho(\theta_x,x) = 1 \]
Here, the Assumption \ref{ArgMaxDis} is crucial. Indeed, we need this assumption to suppose that the function $\theta \mapsto D(\theta,z)$ is continuous. This continuity implies the continuity of $\theta \mapsto \rho(\theta,z)$. For $\theta = 0$ the lefthand side is 0 and for $\theta \to \infty$, since $\max_i \psi_i ^{\theta} \geq n^{-1}$ we have that 

\[\rho(\theta,z) \geq  \max_{i,j}(\sqrt{D(\theta,z)}S\sqrt{D(\theta,z)})_{i,j} \geq n^{-1} (\min_i m_i(z)) (\min_{i,j} \sigma^2_{i,j}) \]

Therefore since $M := n^{-1} (\min_i m_i(x)) (\min_{i,j} \sigma^2_{i,j})$ is such that $\rho(\theta,x) \geq M$, we have $ \theta' \rho(\theta, x) \underset{\theta \to \infty}{\longrightarrow} + \infty$. By continuity, there is at least one $\theta_x$ such that $\theta_x' \rho(\theta_x,x)=1$ and so Theorem \ref{lbBM} is proved. 
\end{proof}

\section{Approximation of continuous and non-negative variance profiles}\label{Approx}

We now choose $\sigma :[0,1]^2 \mapsto \R^{+}$ continuous and symmetric and consider the random matrix model $X_N^{(\beta)} := \Sigma_N \odot W_N^{(\beta)}$ where 

\[ \lim_{N \to \infty} \Big| \Sigma_N(i,j) -  \sigma \left( \frac{i}{N},\frac{j}{N} \right) \Big| =0 \]



In order to prove a large deviation principle for $X_N$, we will approximate the variance profile by a piecewise constant $\sigma^n$. Namely, for $n \in \N$ we let $\sigma^n$ be the following $n \times n$ matrix: 
\[ \sigma^n_{i,j} = n^2  \int_{\frac{i -1}{n}}^{\frac{i}{n}}  \int_{\frac{j -1}{n}}^{\frac{j}{n}} \sigma(x,y) dx dy  + \frac{1}{n+1}.\]

The term $1/(n+1)$ is here to ensure that the approximating variance profile is positive so that Assumption \ref{Pos} is satisfied. Let's denote $X_N^{(\beta),n} = \Sigma^{(n)}_N \odot W_N^{(\beta)}$ the random matrix constructed with the same family of random variables $a^{(\beta)}_{i,j}$ but with the piecewise constant variance profile associated with the matrix $\sigma^n$ and the vector of parameters $(\frac{1}{n},...,\frac{1}{n})$. Let $F^n = F( \sigma^n,.)$, $\mu^n := \mu_{\sigma^n}$. We will also denote $F = F(\sigma,.)$ and $I = I^{(\beta)}(\sigma,.)$. Even if we suppose that Assumption \ref{ArgMaxCont} holds in the case of the continuous variance profile $\sigma$, we don't necessarily have Assumption \ref{ArgMaxDis} for the variance profiles $\sigma^n$ and so we don't necessarily have a sharp lower bound. To this end we will need to introduce an error term $E^n$ that will be negligible as $n$ tends to $\infty$. 

In the first subsection, we will prove that there exist for every $n$ a function from $\R^+$ to itself $E^n$ and a function $\theta \mapsto \psi^{\theta,E^n}$ from $\R^+$ to $\{ x \in  (\R^+)^n : x_1 +... +x_n =1\}$ such that:

\[ \vec{\Psi}(\sigma^n,\theta, \psi^{\theta,E^n})= F^{n}(\theta) - E^{n}(\theta) \]
and such that $\lim_{n \to \infty} \sup_{\theta \geq 0} E^n / \theta^2 = 0$. In the second subsection, we will prove that the upper and lower large deviation bounds we get for $X_N^{(\beta),n}$ from Theorems \ref{theowldubBM} and \ref{lbBM} (which will be denoted respectively $I^{(n)}$ and $\tilde{I}^{(n)}$) both converges toward the rate function defined in section \ref{rate}.

\subsection{Existence of an error negligible toward infinity}

\begin{lemma}
Recalling the definition for $F$: 
\[ F(\theta) = F(\sigma, \theta) = \sup_{\mu \in \mathcal{P}([0,1]) }\left\{ \frac{\theta^2}{\beta} P( \sigma,\mu ) - \frac{\beta}{2} D(\text{Leb}||\mu) \right\}. \]
and recalling that $F^n$ is defined the same way by replacing $\sigma$ by $\sigma^n$, we have the following limit. 
\[ \lim_{n \to \infty} \sup_{\theta >0 } \frac{ |F^{n}(\theta) - F(\theta)|}{\theta^2} = 0 \]

\end{lemma}
\begin{proof}
\begin{multline*}
\Big|\int_0^1 \int_0^1 (\sigma^n)^2(x,y) d \mu(x) d \mu(y) - \int_0^1 \int_0^1 \sigma^2(x,y) d \mu(x) d \mu(y) \Big| \\ \leq \int_0^1 \int_0^1 | (\sigma^n)^2(x,y)  -  \sigma^2(x,y) | d \mu(x) d \mu(y).
\end{multline*}
Since $\lim_{n \to \infty} \sup_{x,y} | (\sigma^n)^2(x,y)  -  \sigma^2(x,y) | =0$, we have $\lim_{n \to \infty} \sup_{\psi} | \vec{P}(\sigma,\psi)- \vec{P}(\sigma^n,\psi)| =0$. The result follows easily by the definitions of $F^n$ and $F$. 
\end{proof}

\begin{lemma}
	Let us recall the definition of $\vec{\Psi}$: 
\[  \vec{\Psi}(\theta,\sigma,\vec{\psi}) := \frac{\theta^2}{\beta} \vec{P}(\sigma, \psi_1,...,\psi_n) + \frac{\beta}{2} \left( \sum_{i=1}^n \alpha_i \log \psi_i - \sum_{i=1}^n \alpha_i \log \alpha_i \right) \]	
If the Assumption \ref{ArgMaxCont} is true, then for every $\epsilon > 0$, there is a sequence of functions $E^n$ and continuous $\theta \mapsto (\psi_i^{\theta,E^n})_{i\in[1,n]}$ such that: 

\[ \vec{\Psi}(\sigma^n,\theta, \psi^{\theta,E^n})= F^{n}(\theta) - E^{n}(\theta) \]

and there is a $n_0$ such that for $n \geq n_0$: 

\[ \forall \theta > 0, E^{n}(\theta) \leq \epsilon \theta^2 .\]

\end{lemma}

\begin{proof}
Since assumption \ref{ArgMaxCont} is verified, there is some measure valued continuous function $\theta \mapsto \psi^{\theta}$ such that $F(\theta) = \theta^2 P(\sigma,\psi^{\theta})/ \beta - \beta D( Leb||\psi^{\theta})/2$. Let $\psi^{\theta,\epsilon} := K_{\star}( \psi^{\theta} * \tau_{\epsilon})$ where $*$ is the convolution, $K_{\star}$ is the push-forward by the application $K$, $\tau_{\epsilon}$ the probability measure whose density is a triangular function of support $[-\epsilon, \epsilon]$ and $K$ the function defined by $K(x) = x $ if $x \in [0,1]$, $K(x) = -x$ if $x \in [-1,0]$ and $K(x)=2-x$ if $x \in [1,2]$. The function $K$ is needed here in order for the support of $\psi^{\theta, \epsilon}$ to still be $[0,1]$ since the convolution with $\tau_{\epsilon}$ enlarges the support to $[-\epsilon, 1+ \epsilon]$. Let us now denote

\[ \psi_i^{\theta,\epsilon,n} := \psi^{\theta,\epsilon} \left( \left[ \frac{i -1}{n}, \frac{i}{n} \right] \right) \]
We have that for $i = 1,..., n$:
\[ \psi_i^{\theta,\epsilon,n} := \int_{\R} (\mathds{1}_{[(i-1)/n , i/n]}+ \mathds{1}_{[-i/n , (1-i)/n]}+ \mathds{1}_{[2-i/n , 2- (1-i)/n]}) (x) d (\psi^{\theta}* \tau_{\epsilon})(x) .\]

So for $i =1, ..., n$:
\[ \psi_i^{\theta,\epsilon,n} := \int_{\R} (\mathds{1}_{[(i-1)/n , i/n]}+ \mathds{1}_{[-i/n , (1-i)/n]}+ \mathds{1}_{[2-i/n , 2- (1-i)/n]})* \tau_{\epsilon} (x) d \psi^{\theta}(x) \]

Since  $x \mapsto (\mathds{1}_{[(i-1)/n , i/n]}+ \mathds{1}_{[-i/n , (1-i)/n]}+ \mathds{1}_{[2-i/n , 2- (1-i)/n]})* \tau_{\epsilon} (x)$ is continuous and $\theta \mapsto \psi^{\theta}$ is continuous for the weak topology then $\theta \mapsto \psi_i^{\theta,\epsilon,n}$ is continuous for $i = 1,..., n$. 

Let us prove the following lemma:

\begin{lemma}
For every $\eta > 0$, there is $\epsilon >0, n_0> 0$ such that for every $\theta > 0, n \geq n_0$

\[ \Big|\sum_{i,j=1}^n (\sigma^n_{i,j})^2 \psi_i^{\theta,\epsilon,n} \psi_j^{\theta,\epsilon,n} - \int_0^1 \int_0^1 \sigma^2(x,y) d \psi^{\theta}(x) d \psi^{\theta}(y) \Big| \leq \eta \]
and,

\[ \frac{1}{n} \sum_i^n \left( \log \psi_i^{\theta,\epsilon,n} - \log n \right) \geq   - D( Leb || \psi^{\theta}). \]

\end{lemma}
\begin{proof}[Proof of the lemma]
Let $\eta > 0$ and let us find $\epsilon > 0$ such that: 

\[ |\int_0^1 \int_0^1 \sigma^2(x,y) d \psi^{\theta,\epsilon}(x) d \psi^{\theta,\epsilon}(y) - \int_0^1 \int_0^1 \sigma^2(x,y) d \psi^{\theta}(x) d \psi^{\theta}(y)| \leq  \eta .\]

Let us take $X,Y,U_{\epsilon}$, $V_{\epsilon}$ independent random variables of law respectively, $\psi^{\theta}, \psi^{\theta}$ and $\tau_{\epsilon},\tau_{\epsilon}$. Then we have
\begin{multline*} |\int_0^1 \int_0^1 \sigma^2(x,y) d \psi^{\theta,\epsilon}(x) d \psi^{\theta,\epsilon}(y) - \int_0^1 \int_0^1 \sigma^2(x,y) d \psi^{\theta}(x) d \psi^{\theta}(y)| = \\
|\E[\sigma^2( K(X+U_{\epsilon}),K(Y+V_{\epsilon})) - \sigma^2( X,Y)]|.
\end{multline*}
Using the uniform continuity of $\sigma^2$, and that $|K(X+U_{\epsilon}) - X|, |K(Y+V_{\epsilon}) - Y| \leq \epsilon $ almost surely, we have that there exists an $\epsilon > 0$ such that the difference is smaller than $\eta$. This bound does not depend on $\theta$.

Now, let us find $n_0$ such that for $n\geq n_0$, 

\[ |\sum_{i,j=1}^n (\sigma^n_{i,j})^2 \psi_i^{\theta,\epsilon,n} \psi_j^{\theta,\epsilon,n} - \int_0^1 \int_0^1 \sigma^2(x,y) d \psi^{\theta,\epsilon}(x) d \psi^{\theta,\epsilon}(y)| \leq \eta .\]

We have 
\begin{multline*}
|\sum_{i,j=1}^n (\sigma^n_{i,j})^2 \psi_i^{\theta,\epsilon,n} \psi_j^{\theta,\epsilon,n} - \int_0^1 \int_0^1 \sigma^2(x,y) d \psi^{\theta,\epsilon}(x) d \psi^{\theta,\epsilon}(y)| \\ 
\leq \int_0^1 \int_0^1 |(\sigma^n(x,y))^2 - \sigma^2(x,y)| d \psi^{\theta,\epsilon}(x) d \psi^{\theta,\epsilon}(y)
\end{multline*}
where we recall that $ (x,y) \mapsto \sigma^n(x,y)$ is the discretized version of $\sigma$. There again, using the uniform continuity of $\sigma$, we have for every $\epsilon > 0$ the existence of $n_0$ such that for $n \geq n_0$, for all $x,y \in [0,1], |(\sigma^n(x,y))^2 - \sigma^2(x,y)|\leq \eta$. Combining these two inequalities we get the first point. 
Then let us show that:

\[ - D( Leb ||\psi^{\theta,\epsilon}) \geq-  D( Leb ||\psi^{\theta}) .\]
Let $f_{\epsilon}(x) = \max \{ 0, \epsilon^{-1} - \epsilon^{-2} |x| \}$ and 

\[ g_{\epsilon}(x,y) =  f_{\epsilon}(x - y ) + f_{\epsilon}(y +x ) + f_{\epsilon}(2 - x + y ). \]
We have that: 

\[ \frac{d \psi^{\theta,\epsilon}}{dx}(x) = \int_{[0,1]} g_{\epsilon}(x,y) d \psi^{\theta}(y) . \]
Let us notice that $\int_{[0,1]} g_{\epsilon}(x,y) dy = \int_{[0,1]} g_{\epsilon}(y,x) dy =1$. We have

\begin{align*}
 - D( Leb ||\psi^{\theta,\epsilon})& = \int_0^1 \log \left( \frac{d \psi^{\theta,\epsilon}}{dx} \right) dx \\
& = \int_0^1 \log \left( \int_{0}^1  g_{\epsilon}(x,y) d \psi^{\theta}(y) \right) dx \\
&\geq \int_0^1 \log \left( \int_{0}^1 g_{\epsilon}(x,y) \frac{d \psi^{\theta}(y)}{dx} dy \right) dx \\
&\geq \int_0^1  \int_{0}^1  g_{\epsilon}(x,y) \log \left(\frac{d \psi^{\theta}(y)}{dx}  \right)dy dx \\
& \geq - D( Leb || \psi^{\theta}) \\
\end{align*}
 where we used the concavity of $\log$. 
 Finally, using again the concavity, we have for every $i \in [1,n]$

\begin{align*}
 n \int_{(i-1)/n}^{i/n} \log \left( \frac{d \psi^{\theta,\epsilon}(x)}{dx} \right) dx &\leq \log \left(   n \int_{(i-1)/n}^{i/n} \frac{d \psi^{\theta,\epsilon}(x)}{dx} \right) \\
& \leq \log \left(   n \psi_{i}^{\epsilon,\theta} \right).
\end{align*}
Summing over $i$ gives us the result.


\end{proof}
Thererefore, using this lemma for $\epsilon > 0$, there is $\epsilon' > 0$ such that

\[ \inf_{\theta \geq 0} \frac {  \vec{\Psi}(\theta,\sigma^n, \psi^{\theta,\epsilon',n}) -F(\theta)  }{\theta^2} >  - \epsilon \]
 for $n$ large enough and so: 
\[ \inf_{\theta \geq 0} \frac {  \vec{\Psi}(\theta,\sigma^n, \psi^{\theta,\epsilon',n}) - F^n(\theta)  }{\theta^2} >  - \epsilon .\]

For $n$ large enough. Therefore, taking 
\[ E^{n}(\theta) := F^n(\theta) - \vec{\Psi}(\theta,\sigma^n, \psi^{\theta,\epsilon',n}) \]
our result is proven.
\end{proof}

\subsection{Convergence of large deviation bounds toward the rate function}

We can now introduce $I_{\beta}^{n}$ and  $\tilde{I}_{\beta}^{n}$ defined on $[r_{\sigma},+ \infty[$ the rate functions for the upper and lower bound of the piecewise constant approximations

\[ I_{\beta}^{n}(x) := \sup_{\theta\geq 0} \left[ J(\mu_{\sigma^n},\theta,x) - F^n(\theta)  \right] \]
\[ \tilde{I}_{\beta}^{n}(x) := \sup_{\theta \geq 0} \left[ J(\mu_{\sigma^n},\theta,x) - F^n(\theta) + E^n(\theta) \right]. \]
To prove that those two functions converge toward $I(x)= I^{(\beta)}(\sigma,.)$ we will need the following result:

\begin{lemma}\label{conv}
	We recall that by definition $\mu^n = \mu_{\sigma^n}$ and $\mu = \mu_{\sigma}$. Then 

\[ \lim_{n\to \infty} \mu^{n} = \mu \]
and if we denote $r^{(n)}$ the upper bound of the support of $\mu^n$ and $l^{(n)}$ its lower bound, we have: 

\[  \lim_{n\to \infty} r^{(n)} = r_{\sigma} \]

\[  \lim_{n\to \infty} l^{(n)} = l_{\sigma} .\]

\end{lemma}

\begin{proof}

The first point is a consequence of Theorem $\ref{Stab1}$. Let $\Delta^{n}_N := X^{(\beta)}_N - X_N^{(\beta),n}$. We have $ \Delta^{n}_N := (\Sigma_N -  \Sigma^{(n)}_N) \odot W^{(\beta)}_N$. Using Lemma \ref{exptight2} and the fact that 
\[ \limsup_N \max_{i,j} |(\Sigma_N -  \Sigma^{(n)}_N)_{i,j}| \underset{n \to \infty}{\longrightarrow} 0 \]
we have that for every $\epsilon >0$ there is $n_0$ such that for any $n \geq n_0$: 

\[ \Pp[ || \Delta^{n}_N || > \epsilon ] \underset{N \to \infty}{\longrightarrow} 0 \]

In particular if  we denote $\lambda_{1,N}< ...<\lambda_{N,N}$ the eigenvalues of $X^{(\beta)}_N$ and $\lambda_{1,N}^{(n)}< ...< \lambda_{N,N}^{(n)}$ these of $X^{(\beta),n}_N$, on the event $\{ || \Delta^{n}_N || \leq  \epsilon \}$ we have $\max_{i=1}^N | \lambda_{i,N} - \lambda_{i,N}^{(n)} | \leq \epsilon$. And so, on this event, we have that for any $a \in \R$: 
\[\# \{ i \in \llbracket 1, N \rrbracket : \lambda^{(n)}_{i,N} < a - \epsilon \} \leq \# \{ i \in \llbracket 1, N \rrbracket : \lambda_{i,N} < a \} \leq \# \{ i \in \llbracket 1, N \rrbracket : \lambda^{(n)}_{i,N} < a + \epsilon \} \]
If we denote for $t \in \R$, $F(t) := \mu( ] - \infty, t])
$ and $F^{(n)}(t) := \mu^n( ] - \infty, t])
$, using the convergence in probability of the eigenvalue distribution, this implies that for every $t$ in $\R$: 

\[ F^{(n)}( t - 2 \epsilon) \leq F(t) \leq F^{(n)}(t + 2 \epsilon).\]
 This then easily implies that $|r^{(n)} - r_{\sigma}| \leq 2\epsilon$.

%
%
%
\end{proof}

This result enables us to finally prove the complete version of Theorem \ref{CVlambdamax}. Indeed using the Theorem \ref{convmunBM} we have that for every $\epsilon > 0$, $\Pp[ \lambda_{\max}(X_N^{(\beta)}) \leq r_{\sigma}- \epsilon] = o(e^{-N})$. It suffices to show that for all $\epsilon > 0$, $\Pp[ \lambda_{\max}(X_N^{(\beta)}) \geq r_{\sigma}+ \epsilon] = O( N^{-2})$. In both the continuous and the piecewise constant case that does not satisfy Assumption \ref{Pos}, we can approximate $\sigma$ by $\sigma^n$ positive. And so the results of Lemma \ref{conv} hold, that is for $n$ large enough, we have $r^{(n)} \leq r_{\sigma}+ \epsilon/2$.  For $n$ large enough, we have $\Pp[ ||\Delta_N^{n}||\geq \epsilon /2] = o( \exp (-N))$. So we have: 
\begin{align*}
 \Pp[ \lambda_{\max}(X_N^{(\beta)}) \geq r_{\sigma}+ \epsilon] &\leq \Pp[ ||\Delta_N^{n}||\geq \epsilon /2] + \Pp[ \lambda_{\max}(X_N^{(n)}) \geq r_{\sigma} + \epsilon/2] \\
&= O(N^{-2})
\end{align*}
where we used Theorem \ref{Convergence} for $X_N^n$. And so Theorem \ref{CVlambdamax} is proved.

\begin{lemma}

\begin{itemize}

\item For every $x > r_{\sigma}$, the function $\theta \mapsto  J( \mu^{n}, \theta, x)$ converges uniformly on all compact sets of $\R^+$ towards $\theta \mapsto  J( \mu, \theta, x)$.

\item For every $x > r_{\sigma}$ \[ I_{\beta}^n(x) \to I(x) \]. 
\[ \tilde{I}_{\beta}^n(x) \to I(x). \]
where we recall that $I$ is equal to $I^ {(\beta)}(\sigma,.)$. 
\end{itemize}
\end{lemma}

\begin{proof}








For the first point of the lemma, let's first prove that for every $x \geq r_{\sigma}$, $\theta \mapsto J(\mu^n, \theta, x)$ converges uniformly on every compact towards the function $\theta \mapsto J(\mu, \theta, x)$.  Let $l<r$ be two reals. For $\mu$ a probability measure on $\R$ whose support is a subset of $]l, r[$, let $Q_{\mu}$ be the function defined on $\mathcal{D}_{r,\epsilon}= \{ ( \theta,u) \in \R^+ \times ]r, + \infty[ : \frac{2\theta}{\beta} (r - u) \leq 1 - \epsilon \}$ by 
\[ Q_{\mu}(\theta,u)= \int_{y \in \R} \log \left(  1 + \frac{2 \theta}{\beta} ( u - y) \right) d \mu(y) \]
$Q_{\mu}$ is continuous in $(\theta,u)$ and for $K \subset  \mathcal{D}_{r,\epsilon}$ a compact we have that the function $\mu \to {Q_{\mu}}_{| K}$ mapping $\mu$ to the restriction of $Q_{\mu}$ on $K$ is continuous in $\mu$ for the weak topology and $\mu$ such that their support is a subset of $]l, r[$ when the arrival space is the set of functions on $K$ endowed with the uniform norm (this is a consequence of Ascoli's theorem). 
Let $x >  r_{\sigma}$ and $r,l$ such that $l< l_{\sigma}< r_{\sigma}<r<x$. For $n$ large enough the support of $\mu^n$ is in $]l,r[$. We have that the sequence of functions $\theta \mapsto  v(\theta,\mu^n,x)$  converges to $\theta \mapsto  v(\theta,\mu,x)$. Indeed if $\frac{2 \theta}{\beta} > G_{\mu}(x)$, then since $\lim_{n \to \infty}G_{\mu^n}(x) = G_{\mu}(x)$, $\frac{2 \theta}{\beta} > G_{\mu^n}(x)$ for $n$ large enough and the result is immediate. 

If $\frac{2 \theta}{\beta} < G_{\mu}(x)$ then $\frac{2 \theta}{\beta} < G_{\mu^n}(x)$ for $n$ large enough. $G_{\mu^n}$ converge towards $G_{\mu}$ on $[r, + \infty[$, for $\epsilon > 0$, $K_{\mu^n}$ is defined on $]0, G_{\mu}(r) - \epsilon]$ for $n$ large enough and $K_{\mu^n}$ converges toward $K_{\mu}$ and therefore $R_{\mu^n}(\frac{2\theta}{\beta})$ converges towards $R_{\mu}(\frac{2\theta}{\beta})$. For $\theta = G_{\mu}(x)$ we use 
that $v(\theta, \mu, x) = R_{\mu^n}(\frac{2\theta}{\beta})$ if  $\frac{2 \theta}{\beta} \leq  G_{\mu^n}(x)$ and  $v(\theta, \mu, x) = (x - \frac{\beta}{2\theta})$ if $\frac{2 \theta}{\beta} \geq  G_{\mu^n}(x)$ and that the limits in both cases are $v(\theta,\lambda,x)$.

Then we have that for $ 2 \theta/\beta \leq G_{\mu}(x)$
\begin{align*}
 \frac{2\theta}{\beta} (r -v(\theta,\mu,x))  &= \frac{2\theta}{\beta}( r - R_{\mu}(2\theta/ \beta)) \\
&=  \frac{2\theta}{\beta}( r - K_{\mu}(\beta/2\theta) + \beta/ 2 \theta) \\
& \leq 1 -  \frac{2\theta}{\beta}( r - K_{\mu}(\beta/2\theta)). \\
\end{align*}
Writing $\frac{2\theta}{\beta} = G_{\mu}(y)$ with $y>x$ we have 
\[ \frac{2\theta}{\beta} (r -v(\theta,\mu,x)) \leq 1 - G_{\mu}(y)(y- r ) \leq 1 - G_{\mu}(x)(x- r ) \]
where we used that $ y \mapsto G_{\mu}(y)(y- r )$ is increasing.

For $ 2 \theta/\beta \geq G_{\mu}(x)$
\begin{align*}
 \frac{2\theta}{\beta} (r -v(\theta,\mu,x))  &= \frac{2\theta}{\beta}( r - x) + 1\\
&\leq  1 - G_{\mu}(x)( x - r ).
\end{align*}

Taking $\epsilon >0$ such that $  \epsilon < G_{\mu}(x)(x- r )$ and using the continuity in $\mu$ of $\mu \mapsto G_{\mu}(x)(x- r )$, we have for every compact $K' \subset \R^+$ and $n$ large enough:
\[ \sup_{\theta \in K'}\frac{2\theta}{\beta}(r -v(\theta,\mu^n,x)) \leq 1 - \epsilon .\]
Therefore, using the convergence of $v(\theta,\mu^n,x)$ and the uniform convergence of $Q_{\mu^n}$ on the compacts of $\mathcal{D}_{r, \epsilon}$, since:
 \[ J(\mu^n,\theta,x) = \theta v(\mu^n,\theta,x) - \frac{\beta}{2} Q_{\mu^n}(\theta,v(\mu^n,\theta,x)) \]
we have that $J(\mu^n,\theta,x)$ converges towards $J(\mu_{\sigma},\theta,x)$. Furthermore, since $\theta \mapsto J(\mu_{\sigma},\theta,x)$ are continuous increasing functions, by Dini's theorem the convergence is uniform on all compact.






We now prove the convergence of $I^n$ towards $I$. Let us prove that there is $A> 0$ and $n_0 \in \N$ such that for $n \geq n_0$ and $\theta >0$

\[  F^{n}(\theta) - E^{n}(\theta) \geq A \theta^2 .\]

We have 

\begin{align*}
F^{n}( \theta) &\geq \frac{\theta^2}{\beta} P(\sigma^n,Leb) \\
&\geq \frac{\theta^2}{\beta} \int_0^1 \int_0^1 \sigma^2(x,y) dx dy
\end{align*}
 and $  - E^{n}(\theta) \geq \epsilon \theta^2$ for n large enough. Choosing $\epsilon < \int_0^1 \int_0^1 \sigma^2(x,y) dx dy$ we have our result. 

Then given that $J( \mu^n,\theta,x) \leq \theta \max( r_n,x)$ we have that for any $r > r_{\sigma}, x > r_{\sigma}, \theta > 0$ for $n$ large enough:

\[ J( \mu^n,\theta,x) - F^{n}(\theta) +  E^{n}(\theta) \leq  r\theta x - A \theta^2 \]

Since $\lim_{\theta \to \infty} r \theta x - A \theta^2 = - \infty$. and that $\theta \mapsto J( \mu^n,\theta,x) - F^{n}(\theta) +  E^{n}(\theta) $ converges toward $\theta \mapsto J( \mu_{\sigma},\theta,x) - F(\theta)$ on every compact of $\R^+$, we deduce that for every $x > r_{\sigma}$: 

\[ \lim_{n \to \infty} \tilde{I}_{\beta}^{n}(x) = I(x) \]

and in the same way  with $E^n = 0$:
\[ \lim_{n \to \infty} I_{\beta}^{n}(x) = I(x) .\]

\end{proof}

\subsection{Conclusion}

We will now prove that the difference between $X_N^{(\beta),n}$ and $X^{(\beta)}_N$ is negligible at the exponential scale. 

\begin{lemma}
For every $\epsilon >0$ and every $A>0$, there exists some $n_0 \in \N$ such that for $n \geq n_0$: 
\[ \limsup_{N} \frac{1}{N} \log  \Pp [ || X_N^{(\beta),n} - X^{(\beta)}_N || \geq \epsilon]  \leq - A \]

\end{lemma}

\begin{proof}
We can write that 

\[ X_N^{(\beta),n} - X^{(\beta)}_N = \Delta_N^{n} \odot W_N^{(\beta)} \]
where 
 \[  \Delta_N^n = \Sigma_N^{(n)} - \Sigma_N \]
Let
\[ M_n := \sup_{i,j} |(\Delta_N^n)_{i,j}| \]

We have that: 
\[ \lim_{n \to \infty} M_n = 0 \]
Following Lemma \ref{exptight2}, we can write that for every $n \in \N$, $A >0$ there is $B >0$ such that 

\[ \limsup_N \frac{1}{N} \log \Pp[ (M_{n})^{-1} ||X_N^{(\beta),n} - X^{(\beta)}_N|| > B ] \leq - A .\]
For $n_0 \in\N$ such that $M_n B \leq \epsilon$ for all $n \geq n_0$, our upper bound is verified. 

\end{proof}

Therefore, since both $I_{\beta}^{(n)}(x)$ and $\tilde{I}^{(n)}_{\beta}(x)$ converge toward $I_{\beta}(x)$, we have a weak large deviation principle with rate function $I_{\beta}$ . Furthermore since we also have  exponential tightness, we have that Theorem \ref{maintheowBM} holds.

It only remains to relax the positivity assumption \ref{Pos} for the piecewise constant case. Let $\sigma$ be a piecewise variance profile. We can approximate $\sigma$ by $\sigma^n := \sqrt{ \sigma^2 + \frac{1}{n+1}}$. We notice then that with this choice of $\sigma^n$:
\[ \vec{\Psi}(\theta,\sigma^n,\vec{\psi}) = \frac{\theta^2}{(n+1)\beta} + \vec{\Psi}(\theta,\sigma ,\vec{\psi}) \]
 so that if \ref{ArgMaxDis} is verified for $\sigma$, it is verified for $\sigma^n$. And so, as we have just done for the continuous case, we can prove the same way that the rate functions $I(\sigma^n,.)$ converges to $I(\sigma,.)$ and that the large deviation principle holds with $I(\sigma,.)$.

\section{The case of matrices with $2 \times 2$ block variance profiles} \label{2by2}

In this section, we will discuss the case of piecewise constant variance profiles with 4 blocks (which are not necessarily of equal sizes) and determine what are the cases where the Assumption \ref{ArgMaxDis} holds. In particular, we will provide examples where the maximum argument of Assumption \ref{ArgMaxDis} can be taken continuous without the need for the concavity assumption. 

Let's take a piecewise constant variance profile defined by $\vec{\alpha} = ( \alpha, 1- \alpha)$ and $\sigma_{1,1} = a, \sigma_{2,2} = b, \sigma_{1,2} = \sigma_{2,1}= c$. In order to apply Theorem \ref{maintheowBM} we need to study the maximum argument for $\theta$ fixed of: 
\[ \psi(x, \theta) = \vec{\Psi}(\sigma,\theta,(x,(1-x)))= \frac{\theta^2}{\beta} [ a^2 x^2 + b^2 (1- x)^2 + 2 c^2 x (1- x) ] + \frac{\beta}{2} [ \alpha \log x + (1- \alpha) \log(1 - x)] .\]

Since we can change $\alpha$ in $1 - \alpha$ by switching $a$ and $b$, we can suppose without loss of generality that $\alpha \leq  1/2$.

We have 

\[ \partial_x \psi(x, \theta) := \frac{2\theta^2}{\beta} [ a^2 x  - b^2 (1- x) +  c^2 (1 - 2x) ] + \frac{\beta}{2} \left(\frac{ \alpha}{x} -  \frac{1- \alpha} {1 - x} \right) .\]

Let $\Phi(x,\theta) := x(1-x)\partial_x \psi(x, \theta)$, so that $\Phi(.,\theta)$ vanishes at the critical points of $\psi(.,\theta)$. We have that: 
\[ \Phi(x,\theta)  = (a^2 + b^2 - 2 c^2) \frac{2\theta^2}{\beta}x(x - x_{min})(1-x) + \frac{\beta}{2}(\alpha -x) .\]
where

\[ x_{\min} := \frac{ c^2 - b^2 }{ 2 c^2 - a^2 - b^2} \]

\[ x_{\theta} = \text{argmax}_{x \in [0,1]} \psi(x,\theta). \]

\subsection{ Case with $ (a^2 + b^2 - 2 c^2) \leq  0$ }

In the case $(a^2 + b^2 - 2 c^2) \leq  0$ we have the $\psi(., \theta)$ is strictly concave and therefore $\theta \mapsto x_{\theta}$ is analytical and assumption \ref{ArgMaxDis} is satisfied and the large deviation principle applies. 

From now on, we assume $ (a^2 + b^2 - 2 c^2) > 0$.

\subsection{ Case $x_{\min} \geq 1/2$ }
 
We look for the zeros of $\Phi(.,\theta)$ in $[0,1]$. To this end, we look for the intersection points of the curve of equation $y = x(1-x)(x- x_{\min})$ and the line $y = A_{\theta}(x - \alpha)$ where $A_{\theta} =  \frac{\beta^2}{4\theta^2(a^2 + b^2 - 2 c^2)}$. 
\begin{figure}
\begin{center}
\includegraphics[scale=8]{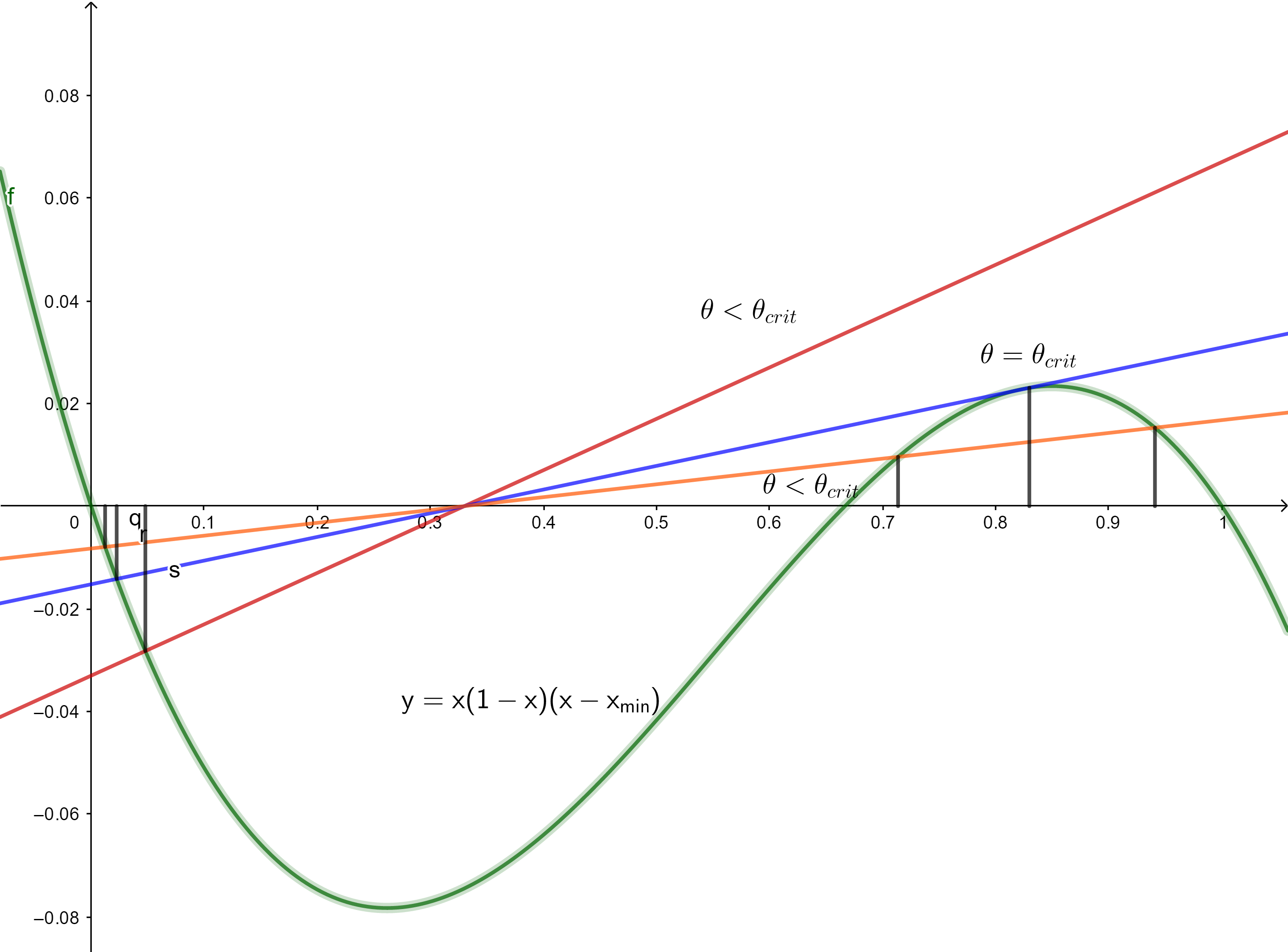}
\caption{ Case $x_{\min} \geq 1/2$ }
\end{center}
\end{figure}

We notice that there is a critical value $\theta_{crit}$ such that for $\theta \leq \theta_{crit}$, there is only one critical point $x_1^{\theta}$ which is on $]0,1/2[$. For $\theta > \theta_{crit}$ we have the apparition of two other critical points $x_2^{\theta}$ and $x_3^{\theta}$ that are such that $1/2 <x_2^{\theta} < x_3^{\theta}$ with $\psi(x_2^{\theta},\theta)$ being a local minimum and $\psi(x_3^{\theta},\theta)$ a local maximum. For $x \in]0,1[$, we have:

\begin{multline*} \psi(x, \theta) - \psi(1-x,\theta) =  \\ \frac{\beta}{2} (1 - 2 \alpha) ( \log (1-x) - \log x)+\frac{\theta^2}{\beta}(a^2 + b^2 - 2 c^2)[ (x - x_{min})^2 - (1 -x - x_{min})^2] .
\end{multline*}

For $x < 1/2$, we have $\psi(x, \theta) > \psi (1-x,\theta)$ and so the absolute maximum of $\psi(.,\theta)$ is attained at $x = x_1^{\theta}$. Furthermore, we notice the line $y = A_{\theta}( x - \alpha)$ is never tangent to the graph $y = x(x - x_{min})(1-x)$ in the point of first coordinate $x_1^{\theta}$, we have $\partial_x \Phi(x_1^{(\theta)},\theta ) \neq 0$. Now using the implicit function theorem, we have $\theta \mapsto x_1^{\theta}$ is analytical and so Assumption \ref{ArgMaxDis} is verified.

\subsection{ Case $x_{min} \leq \alpha$}

\begin{figure}
\begin{center}
\includegraphics[scale=8]{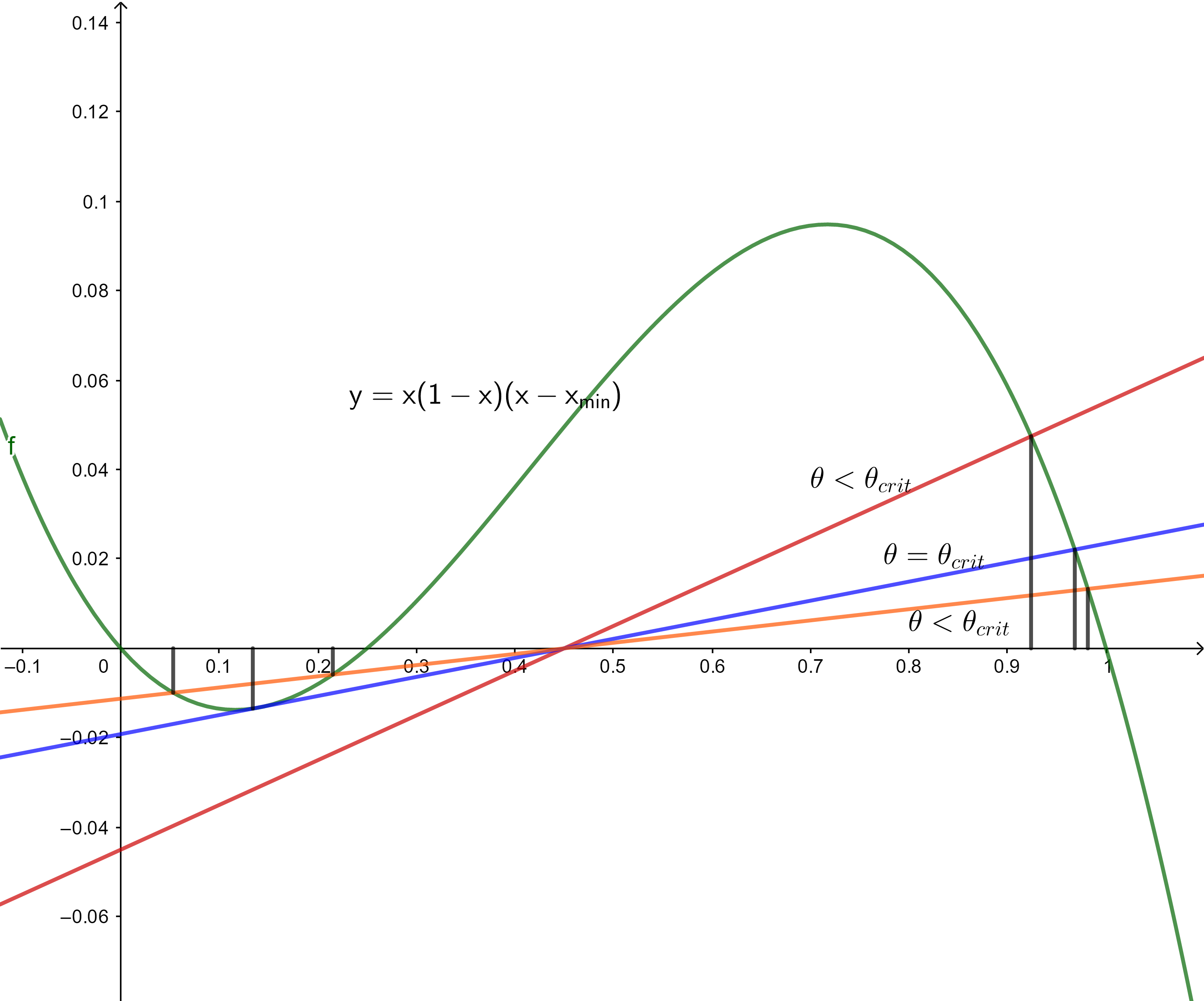}
\caption{Case $x_{min} \leq \alpha$}
\end{center}
\end{figure}
There is again a critical value $\theta_{crit}$ such that for $\theta \leq \theta_{crit}$, there is only one critical point $x_1^{\theta}$ which is on $]\alpha,1[$ and for $\theta > \theta_{crit}$ we have the apparition of two other critical points $x_2^{\theta}$ and $x_3^{\theta}$ such that $x_3^{\theta} < x_2^{\theta}\leq \alpha$. We have furthermore: 

\begin{multline*}
 \psi(x, \theta) - \psi(2 \alpha -x,\theta) = \frac{\beta}{2} [ \alpha \log x - (1- \alpha) \log (1 - 2 \alpha +x)] + \frac{\beta}{2} [ (1 - \alpha) \log(1-x) - \alpha \log (2 \alpha - x) ] \\ + \frac{\theta^2}{2}(a^2 + b^2 - 2 c^2)[ (x - x_{min})^2 - (2 \alpha -x - x_{min})^2] .
\end{multline*}

For $ x< \alpha$, $\psi(x,\theta) < \psi(2 \alpha-x,\theta)$ and there the absolute maximum of $\psi(.,\theta)$ is attained on $]\alpha, 1[$, so for $x_1^{\theta}$. Since $\theta \mapsto x_1^{\theta}$ is analytical, Assumption \ref{ArgMaxDis} is verified.

\subsection{Case $\alpha =1/2$ and $x_{\min} = 1/2$}

Then $x \mapsto \psi(x, \theta)$ is symmetrical in $1/2$. Looking at the zeros of $\Phi(., \theta)$ we have that if we set $\theta_{crit} := \beta \sqrt{2 / ( a^2 + b^2 - 2 c^2)}$ for $\theta \leq \theta_{crit}$ there is only one zero a $x = 1/2$ and for $\theta \geq \theta_{crit}$ there is three zeros in $ x= 1/2$ and $x = \frac{1}{2}  \pm \delta(\theta)$ where $\delta(\theta) = \frac{\beta}{2 \theta} \sqrt{\frac{2}{(a^2 + b^2 - 2 c^2) }}$. Furthermore, for $\theta \leq \theta_{crit}$, $\psi(., \theta)$ has its maximum in $x=1/2$ and for $\theta \geq \theta_{crit}$, $\psi(., \theta)$ has its maximum at both points $x=1/2 \pm \delta(\theta)$ where $\delta(\theta) = \frac{1}{2}\sqrt{1 - \frac{2 \beta^2}{\theta^2(a^2 + b^2 - 2 c^2)}}$. Therefore the function $\theta \mapsto 1/2$ for $\theta \leq \theta_{crit}$ and $\theta \mapsto 1/2 + \delta(\theta)$ for $\theta \geq \theta_{crit}$ is a continuous determination of the maximal argument of $\psi(.,\theta)$ and so Assumption \ref{ArgMaxDis} is verified and the large deviation principle holds. This gives an example where the maximum argument in Assumption \ref{ArgMaxDis} is neither unique nor $\mathcal{C}^1$ but where we can still derive a large deviation principle.

\subsection{ Case $x_{min} \in ]\alpha,1/2[$ and pathological cases}

The graph we obtain is similar to the graph of the first case. In this case, we also have a $\theta_{crit}$ such that for $\theta \leq \theta_{crit}$, there is only one critical point $x_1^{\theta}$ which is in $]0,\alpha[$ and then the apparition of two other critical points $x_2^{\theta}$ and $x_3^{\theta}$ that are such that $1/2 <x_2^{\theta} < x_3^{\theta}$, $\psi(x_2^{\theta},\theta)$ being a local minimum and $\psi(x_3^{\theta},\theta)$ a local maximum. But in this case for high values of $\theta$, we have that the maximum is attained near $1$ and so for these high values $x_3^{\theta}$ is the maximum argument. We have a discontinuity in the maximum argument and so Assumption \ref{ArgMaxDis} is not verified.

Let us now show that if $x_{min} \in ]\alpha,1/2[$ and $c =0$, the largest eigenvalue satisfies a large deviation principle but with a rate function $J$ different from $I$.

Our matrix $X_N^{(\beta)}$ then looks like:
\[ \begin{pmatrix}
\sqrt{a\alpha} T_{\alpha N}^{(\beta)} & 0 \\
0 & \sqrt{b(1- \alpha)} U_{(1-\alpha) N}^{(\beta)} \\
\end{pmatrix} \]

where $(T_N^{(\beta)})_N$ and $(U_N^{(\beta)})_N$ are independent Wigner matrices. We have: 

\[ \lambda_{\max}(X_N^{(\beta)}) = \max \{ \sqrt{a\alpha} \lambda_{\max}(T_{\alpha N}^{(\beta)}), \sqrt{b(1- \alpha)} \lambda_{\max}(U_{(1- \alpha) N}^{(\beta)}) \} .\]

But both these quantities satisfy large deviation principles, more precisely, if $I_{\beta}$ is the rate function for the large deviation principle of \cite{GuHu18} for a Wigner matrix, $\lambda_{\max}(\sqrt{a \alpha} T_{\alpha N}^{(\beta)})$ follows a large deviation principle with rate function $\alpha I_{\beta}( \frac{x}{\sqrt{ a  \alpha}})$ and  $\lambda_{\max}(\sqrt{b(1- \alpha)} U_{(1-\alpha) N}^{(\beta)})$ follows a large deviation principle with rate function$(1-\alpha) I_{\beta}( \frac{x}{\sqrt{ b (1 -\alpha)}})$. Now $\lambda_{\max}(X_N^{(\beta)})$ follows a large deviation principle with rate function $J^{\beta}$ which is: 

\[ J^{\beta}(x) := \min \{ \alpha I_{\beta}( \frac{x}{\sqrt{ a \alpha}}), (1-\alpha) I_{\beta}( \frac{x}{\sqrt{ b (1 - \alpha)}})  \}\]

In particular, if we choose $\alpha, a,b$ such that $a \alpha > b(1- \alpha)$ and $b >a$, we notice that $J^{\beta}(x) = \alpha I_{\beta}( \frac{x}{\sqrt{ a \alpha}})$ for $x$ near $a \alpha$ and  $J^{\beta}(x) = (1-\alpha) I_{\beta}( \frac{x}{\sqrt{ b (1 - \alpha)}})$ for large $x$. In this case one can notice that $J^{\beta}$ is not a convex function and therefore cannot by obtained as $\sup_{\theta} \{ J(x,\mu_{\sigma}, \theta)- F(\theta) \}$ since it is a $\sup$ of convex functions. We have $J \neq I$. 
\begin{figure}
\begin{center}
\includegraphics{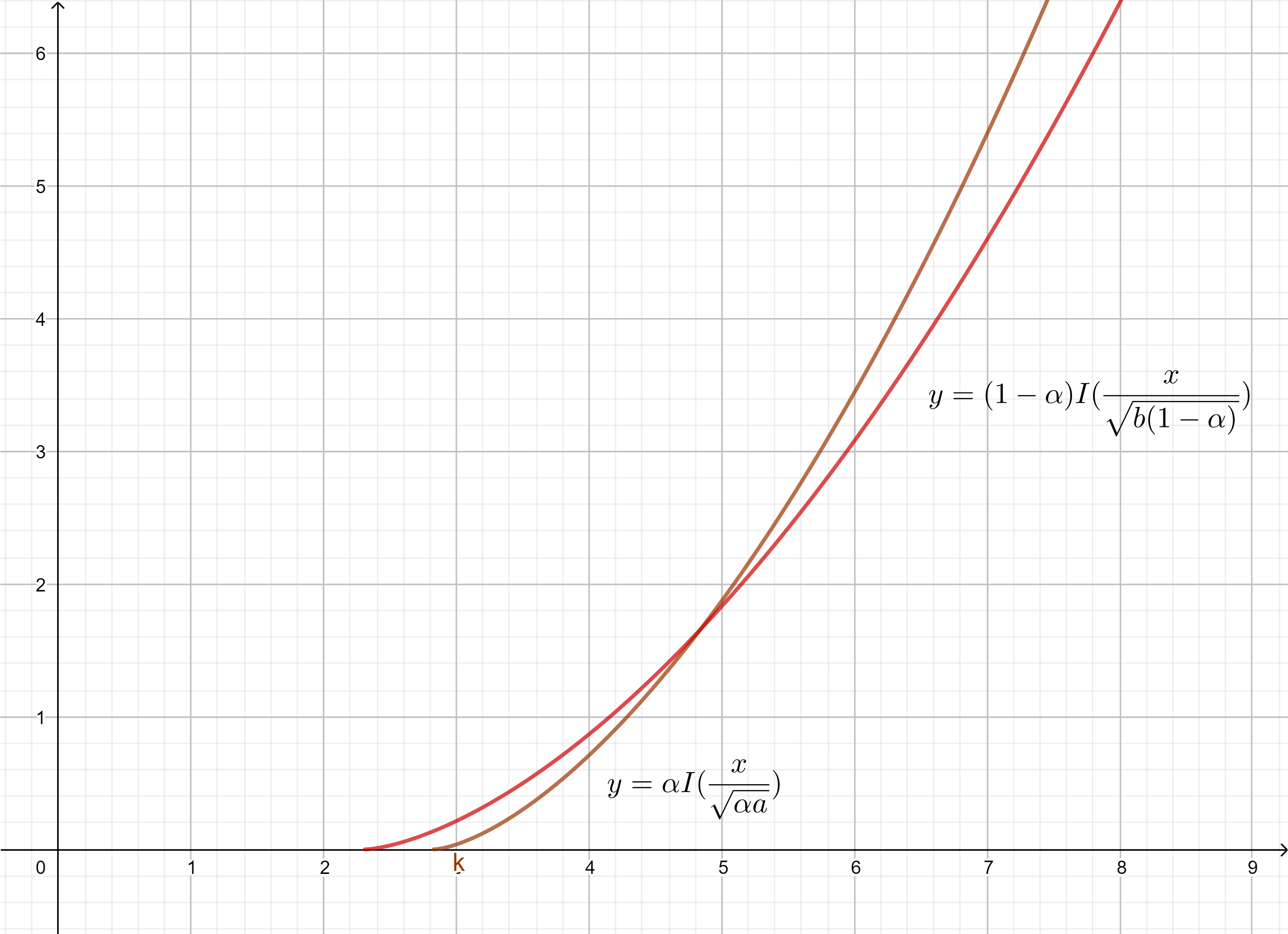}
\caption{ Graph of the rate function $\alpha I(x/ \sqrt{a \alpha})$ and $(1-\alpha) I(x/ \sqrt{b(1 - \alpha)})$ with $\alpha= 2/3$, $a = 3$, $b =4$}
\end{center}
\end{figure}
For $c>0$ but small enough we can also conclude that the large deviation principle still does not hold. Indeed, if we denote $I_c$ the rate function we expect using the formula of section \ref{rate}. Since $I_0$ still provides a large deviation upper bound, we have $J \geq I_0$ and so let $x_0\in \R^+$ such that $J(x_0)\geq  I_0(x_0) + \eta$ for some $\eta>0$ ($x_0$ does exists since $J \neq I_0$). Using the same approximation arguments as in section \ref{Approx}, we have that there exists $\epsilon > 0$ such that for $c < \epsilon$, $ I_c(x_0) < I_0(x_0) + \eta/3 $ and: 

\[ \lim_{\delta \to 0} \limsup_N - \frac{1}{N} \log \Pp[ \lambda_{\max}(X_N) \in [x_0 - \delta, x_0 + \delta]] \geq J(x_0) - \eta/3 \geq I_0(x_0) + 2\eta/3 \geq I_c(x_0) + \eta/3 \]

Since $I_c$ is continuous in $x_0$, we have that there cannot be a large deviation principle with the rate function $I_c$. 
\section{Looking for an expression of the rate function}\label{ratefunc}
In this section we will present a method to explicitly compute the rate function $I$ in the piecewise constant case under some hypothesis on the behavior of $F$. First, let us describe $F$ in a neighbourhood of $\theta = 0$. 

\begin{theo}
Let $\sigma$ be a continuous or piecewise constant variance profile, there is $\theta_0 > 0$ such that for $\theta \leq \theta_0$: 
\[ F(\sigma,\theta) = \frac{\beta}{2} \int_0^{\frac{2}{\beta} \theta} R(w) dw \]
Where $R$ is the $R$-transform of the measure $\mu_{\sigma}$.
\end{theo}
\begin{proof}
This result was proved in the case of a linearisation of a Wishart matrix in section 4.2 of $\cite{GuHu18}$. For the sake of completeness, we will reproduce here this proof. For the lower bound, we have for $M >r_{\sigma}$ and $2\theta/ \beta \leq G(M)$ (where $G$ is the Stieltjes transform of the measure $\mu_{\sigma}$): 

\[ F(\sigma,\theta) \geq \liminf_{N \to \infty} \frac{1}{N}\E_{X_N}[ \mathds{1}_{\lambda_{\max}(X_N) < M } I_N(X_N,\theta)] \geq \frac{\beta}{2} \int_0^{\frac{2}{\beta} \theta} R(w) dw \]
 This is due to the fact that for $2 \theta/ \beta \leq G(x)$, $J(\theta,x, \mu) = \frac{\beta}{2} \int_{0}^{\frac{2 \theta}{\beta}} R(\theta) d\theta$. For the upper bound, we write:

\begin{align*}
\E_{X_N}[ I_N(X_N,\theta)]  & \leq \sum_{n=1}^{+ \infty} \E_{X_N}[  \mathds{1}_{(n-1) M \leq \lambda_{\max}(X_N) < nM} I_N(X_N,\theta)] \\
& \leq  \E_{X_N}[ \mathds{1}_{\lambda_{\max}(X_N) < M} I_N(X_N,\theta)] + \sum_{n=2}^{+ \infty} \E_{X_N}[  \mathds{1}_{(n-1) M \leq  \lambda_{\max}(X_N) < nM} I_N(X_N,\theta)] \\
& \leq  \E_{X_N}[ \mathds{1}_{\lambda_{\max}(X_N) < M} I_N(X_N,\theta)] + \sum_{n=2}^{+ \infty} \exp( -K(n-1)M + nNM \theta) \\
& \leq  \E_{X_N}[ \mathds{1}_{\lambda_{\max}(X_N) < M} I_N(X_N,\theta)] + e^{N\theta M}\frac{\exp( M (\theta - K)N)}{ 1- \exp((M\theta - K)N)}
\end{align*}
Where we used that for $N$ large enough, we have for every $N$, $\Pp[ \lambda_{\max}(X_N) \geq M] \leq \exp(-KM)$ for some $K>0$ and that for $\lambda_{\max}(X_N) \leq M$, $I_N(X_N,\theta) \leq e^{M\theta}$. Now, by choosing $\theta$ small enough such that $( 2 \theta - K) < 0$, we have the upper bound. 
\end{proof}

The main results of this section is the following: 
\begin{theo}\label{thRateF}
If the function $\theta \mapsto F(\sigma, \theta)$ is analytic, then the $R$ transform of $\mu_{\sigma}$ has an analytic extension on $\R^+$ and then the rate function $I(\sigma,.)$ only depends on $\mu_{\sigma}$. 
\end{theo}
\begin{proof}
Since $F(\sigma,.)$ is analytic and so is $R$ and since we have $F'(\sigma, \theta) = R(\frac{2 \theta}{\beta})$ for small $\theta$, $F$ depends only on $R$ that is on $\mu_{\sigma}$ and $F'( \frac{ \beta x}{2})$ extends $R$ on $\R^+$. Then, looking at the expression of $I(\sigma,.)$, it only depends on $\mu_{\sigma}$. 
\end{proof}
\begin{rem}
Without any condition on the variance profile $\sigma$, we do not have that $I(\sigma,.)$ depends only on $\mu_{\sigma}$. For instance if we take $X_N$ and $X'_N$ independent matrices both with the same variance profile $\sigma$, $\alpha,\beta >0$ such that $\alpha > \beta$ and $\alpha + \beta = 1$, then the following matrix has a variance profile: 

\[ Y_N = \begin{pmatrix} X_{\alpha N} & 0 \\
0 & X'_{\beta N} 
\end{pmatrix}.
\]
And then $\lambda_{\max}(Y_N) = \max( \lambda_{\max}(X_{\alpha N}), \lambda_{\max}(X_{\beta N}))$. We have that $\lambda_{\max}(Y_N)$ satisfy a large deviation principle with rate function $\beta I(\sigma,.)$ whereas this matrix has for limit measure $\mu_{\sigma}$ whatever the choice of $\beta$.  
\end{rem}
In the case of a piecewise constant variance profile, the same concavity hypothesis as before implies the analyticity of the function $F(\sigma,.)$ (this is due to the fact that with the implicit function theorem, the maximum argument is indeed analytic in $\theta$). 
\begin{prop}
If the Assumption \ref{ConcaveDis} holds in the case of a piecewise constant case then the function $\theta \mapsto F(\sigma,\theta)$ is analytic. 
\end{prop}

We will now shortly discuss how we can obtain an explicit expression of the rate function in the context of a piecewise constant variance profile which satisfies the hypothesis of Theorem \ref{thRateF}. For this, we will need the following proposition: 

\begin{prop}
If the hypothesis of Theorem \ref{thRateF} are verified and if $\theta \mapsto R(\theta) + \frac{1}{\theta}$ is strictly increasing on $[G(r_{\sigma}), + \infty[$, then we have: 
\[I(\sigma,x) = \frac{\beta}{2} \int_{r_{\sigma}}^x ( \overline{G}(u) - G(u)) du .\]
where we analytically extended $R$ on $\R^+$, where $G(r_{\sigma}) = \lim_{x \to r_{\sigma}^+} G(x)$ and where $\overline{G}$ is the inverse function of $\theta \mapsto R(\theta) + \frac{1}{\theta}$ on $[r_{\sigma}, + \infty[$. 
\end{prop}
\begin{proof}
for $ x \geq r_{\sigma}$, we have that: 
\[ J(\mu_{\sigma},\theta,x) - F(\sigma,\theta) = \begin{cases} 0 \text{ for } 2\theta/\beta \leq  G(x) \\
 \theta x - \frac{\beta}{2} - \frac{\beta}{2} \Big( \log \Big(  \frac{2 \theta}{\beta} \Big)\Big) - \frac{\beta}{2} \int \log ( x- y) d\mu_{\sigma}(y) - \frac{\beta}{2} \int_0^{\frac{2\theta}{\beta}} R(w) dw \text{ for } 2\theta/\beta \geq  G(x)\\
\end{cases} . \] 
Differentiating in $\theta$, we have: 
\[ \frac{\partial}{\partial \theta} \Big( J(\mu_{\sigma},\theta,x) - F(\sigma,\theta) \Big) = 
\begin{cases} 0 \text{ for } 2\theta/\beta \leq  G(x) \\
x - \frac{\beta}{2 \theta} - R \Big( \frac{2\theta}{\beta} \Big) \text{ for } 2\theta/\beta \geq  G(x) \\
\end{cases} .\]
And so, the maximum is realized for $\theta_x > \beta G(x)/2$ such that $ x = \frac{\beta}{2 \theta_x} + R \Big( \frac{2\theta_x}{\beta} \Big)$. By hypothesis, this is equivalent to $\theta_x = \frac{\beta}{2} \overline{G}(x)$. And so we have for $x > r_{\sigma}$
\[ \frac{\partial}{\partial x} I(\sigma,x) = \theta_x - \frac{\beta}{2}G(x) \]
We deduce the result by integrating. 
\end{proof}

\begin{rem}
In practice, in the case of a piecewise constant variance profile the equations of section \ref{Emp} give that $G(z)$ is a algebraic function, that is a root of some polynomial $P(.,z)$. So we have, for $\theta \leq G( r_{\sigma})$, $P(\theta, R(\theta)+ \theta^{-1}) = 0$. Using the analytical extension of $R$ on $\R^+$, this stays true for any $\theta> 0$ and therefore $P( \overline{G}(z),z)=0$. So $\overline{G}$ naturally presents itself as a conjugate root of $G(z)$. For example, in the Wigner case we have $G(z) = \frac{z- \sqrt{z^2 - 4}}{2}$ and $\overline{G}(z) = \frac{z+ \sqrt{z^2 - 4}}{2}$, and we we end up with $I(x) = \frac{\beta}{2} \int_2^x \sqrt{u^2 - 4} du $. In the case of the linearisation of of Wishart matrix (see section 4.2 of \cite{GuHu18}), we have:

\[ G(x) = \frac{2 \alpha}{1+\alpha} \frac{x^2 +1 - \alpha - \sqrt{ (x^2- 1 -\alpha)^2 - 4 \alpha }}{2x} + \frac{1-\alpha}{1+ \alpha} \frac{1}{x} \]

and  

\[ \overline{G}(x) = \frac{2 \alpha}{1+\alpha} \frac{x^2 +1 - \alpha + \sqrt{ (x^2- 1 -\alpha)^2 - 4 \alpha }}{2x} + \frac{1-\alpha}{1+ \alpha} \frac{1}{x} \]
 and so we have $I(x) = \frac{\beta \alpha}{1+ \alpha} \int_{r_{\sigma}}^x \frac{\sqrt{ (u^2- 1 -\alpha)^2 - 4 \alpha }}{2u}du $.
\end{rem}

\appendix

\section{Appendix: The limit of the empirical measure} \label{Emp}
	In this section, we describe the limit of the empirical measure $\mu_{\sigma}$ of the matrices $X_N$. We will also discuss the stability of this measure in function of $\sigma$. Under assumptions of positivity for the variance profile, we will prove that the largest eigenvalue converges toward the rightmost point of the support of $\mu_{\sigma}$. We denote $\mathbb{H}^+$ the complex upper half-plane $\{ z \in \C : \Im z > 0 \}$ and $\mathbb{H}^-$ the complex lower half-plane $\{ z \in \C : \Im z < 0 \}$.
To describe the limit of the empirical measure we need the following definition for the so-called canonical equation (also called quadratic vector equation). The following definition takes into account both the piecewise constant and the continuous case:

\begin{defi}
Let $ \sigma :[0,1]^2 \to \R^+$ be a bounded symmetric  measurable function. We call canonical equation $K_{\sigma}$ the following equation where $m$ is a function from $\mathbb{H}^+$ into $\mathcal{H}$, where $\mathcal{H}$ is the set of measurable and bounded functions $m$ from $[0,1]$ to $\mathbb{H}^-$,

\begin{equation}\label{Ksigma}
\tag{$K_{\sigma}$}
 \frac{1}{m(z)} = z - Sm(z)
\end{equation}.

Here $S$ is the following kernel operator for $ f \in \mathcal{H}$: 

\[ S f(x) := \int_0^1 \sigma^2(x,y) f(y) dy \].

\end{defi}

If $w$ is a function from $[0,1]$ to $\C$, we denote $||w|| = \sup_{x \in[0,1]}|w_x|$. If $S$ is an operator on the space of functions from  $[0,1]$ to $\C$, we denote $||S||$ operator norm corresponding to the previous norm. If $m$ is a function from $\mathbb{H}^+$ to $\mathcal{H}$, we denote $m_x$ the function $z \mapsto m(z)(x)$. We then have the following result concerning the solution of the equation: 

\begin{theo}
The equation $K_{\sigma}$ has a unique solution $m$ which is analytic in $z$. Moreover for every $x \in [0,1]$, the function $m_x = m(.)(x)$ is the Stieltjes transform of a probability measure $p_x$ on $\R$.

\end{theo}
This theorem is in fact a direct application of Theorem 2.1 from \cite{Aja15} which states that the equation $K_{\sigma}$ always has a solution in a more general context where we replace $[0,1]$ by a probability space $\mathfrak{X}$ and $S$ is a symmetric, positivity preserving operator on the space of uniformly bounded complex functions on $\mathfrak{X}$. 
 
\begin{rem}
If $\sigma$ is a piecewise constant variance profile with parameters $\alpha_1,...,\alpha_n$ and $(\sigma_{i,j})_{1 \leq i,j \leq n}$, then the solution of $(K_{\sigma})$ is piecewise constant on the intervals $I_{j}$. This can be viewed directly from $K_{\sigma}$ by noticing that $Sf$ is always piecewise constant on the intervals $I_{j}$. 
\end{rem}
We will denote $\mu_{\sigma} := \int_0^1 p_x dx$ where $p_x$ is given by the preceding theorem. And so we have that the Stieltjes transform of $\mu_{\sigma}$ is $\int_0^1 m_x dx$. This measure $\mu_{\sigma}$ will be the limit of the empirical measures $\mun$ of the matrices $X_N^{(\beta)}$. To prove this, we will use the following result which is a reformulation of Girko's result \cite[Theorem 1.1]{Girko2001}. 







\begin{theo} \label{Girko}
Let us denote $\sigma_N : [0,1]^2 \to \R^+$ the function $(x,y) \mapsto \Sigma_N(\lceil Nx \rceil , \lceil Nx \rceil)$.
When $N$ tends to infinity, for almost every $x$ we have that with probability 1: 

\[ \lim_{N \to \infty} \Big| \mun (] - \infty, x]) - \int_{- \infty}^x d \mu_{\sigma_N}(x) \Big| = 0 \].

\end{theo}
\begin{proof}

First, since for all $N$, $\Sigma_N$ is bounded and since the coefficients of $X_N^{(\beta)}$ are centered, hypothesis (1.1) and (1.2) are verified. Then, since we have a sharp sub-Gaussian bound on the entries of $X_N^{(\beta)}$, we can easily verify the Lindeberg's condition (1.3). 
Furthermore, since $\sigma_N$ is piecewise constant, the solution $m$ of $K_{\sigma_N}$ is piecewise constant on the intervals $[i/N, (i+1)/N]$. Making the change of variables $c_i(z) = m_{(2i -1)/2}(z)$ we have that the equation $K_{\sigma_N}$ is equivalent, up to the sign since our Stieltjes transform convention is different, to the matricial equation given in \cite[Theorem 1.1]{Girko2001}.
\end{proof}

And so we are left with determining the convergence of the measure $\mu_{\sigma_N}$ when $N$ tends to $+\infty$. To that end, we will need the following rough stability results. 

\begin{theo}\label{Stab1}
Let $\sigma :[0,1]^2 \to \R^+$ be a bounded measurable function. For every open neighbourhood $\mathcal{V}$ of $\mu_{\sigma}$ in $\mathcal{P}([0,1])$ for the vague topology, there is $\eta > 0$ such that for every $\tilde{\sigma}$ bounded measurable function such that $\sup_{x,y \in [0,1]^2} | \sigma^2(x,y) - \tilde{\sigma}^2(x,y)| \leq \eta$, $\mu_{\tilde{\sigma}} \in \mathcal{V}$.
\end{theo}
\begin{proof}
This proof is inspired from the proof of \cite[Proposition 2.1]{AEK2}. We let $\tilde{S}$ be the kernel operator with kernel $\tilde{\sigma}^2$. Let $\mathbb{H}^+_{\eta} = \{ z \in \mathbb{C} : \Im z \geq \eta, |z| \leq \eta^{-1} \}$, $\mathbb{H}^-_{\eta} = \{ z \in \mathbb{C} : \Im z \leq - \eta, |z| \leq \eta^{-1} \}$ and $D$ the function defined on $(\mathbb{H}^-)^2$ by 
 \[ D( \zeta, \omega) = \frac{ | \zeta - \omega|^2}{\Im \zeta \Im \omega} \]
then $d := \arcosh ( 1 + D)  $ is the hyperbolic distance on $\mathbb{H}^-$. For $u$ a function from $\mathbb{H}_{\eta}^+$ to $ \mathcal{H}$ we let $\Phi(u)$ and $\tilde{\Phi}(u)$ be functions from $\mathbb{H}_{\eta}^+ $ to $\mathcal{H}$ defined as follows for $z \in \mathbb{H}_{\eta}^+$: 
\[ \Phi(u)(z) :=  \frac{1}{ z -S u(z)} \]
\[ \tilde{\Phi}(u)(z) :=  \frac{1}{ z - \tilde{S} u(z)} \].

If $\mathcal{B}_{\eta} := \{ u : \mathbb{H}^+_{\eta} \to \mathcal{H} : \sup_{x \in [0,1], z \in \mathbb{H}^+_{\eta}} \Im u_x(z) \leq  -\frac{ \eta^3}{ (2 + \min\{||S||,||\tilde{S}||\})^2},  \sup_{x \in [0,1],z \in \mathbb{H}^+_{\eta}} | u_x(z) | \leq \eta^{-1}\}$, then following the proof of \cite[Proposition 2.1]{AEK2}, if $u$ is such that for all $z \in \mathbb{H}^+_{\eta} $, and all $x \in [0,1]$, $ u_x(z) \in \mathbb{H}_{\eta}^-$, then $\Phi(u) \in \mathcal{B}_{\eta}$. Then, if $\eta < 2 + \min\{||S||,||\tilde{S}||\}$, $\frac{ \eta^3}{ (2 + \min\{||S||,||\tilde{S}||\})^2} \leq \eta$, $\Phi$ maps $\mathcal{B}_{\eta}$ onto itself and so on for $\tilde{\Phi}$. 

For $x \in [0,1]$, $z \in \mathbb{H}^+_{\eta}$ and $\delta \geq \sup_{x,y} |\sigma^2(x,y) - \tilde{\sigma}^2(x,y)|$

\begin{align*}
D( \Phi(u)_x(z),\tilde{\Phi}(u)_x(z)) &= D( (Su)_x(z) - z,  (\tilde{S}u)_x(z) -z) \\
&\leq \frac{ |(S - \tilde{S})u_x(z)|^2 }{\Im z^2} \\
&\leq \frac{ \delta^2 \sup_{x\in [0,1]} |u_x(z)|^2}{\eta^2} \\
&\leq \frac{\delta^2}{\eta^4}.
\end{align*}

Let $m$ be the solution of $K_{\sigma}$, that is the fixed point of $\Phi$. For every $n \in \N$ let $v^{(n)} = \tilde{\Phi}^{(n)}(m)$. We have for $z \in\mathbb{H}_{\eta}^+$: 

\[ \sup_{x \in[0,1]} D( m_x(z) , v^{(1)}_x(z)) \leq \frac{\delta^2}{\eta^4} \]
and following again \cite{AEK2},

\begin{align*}
\sup_{x \in[0,1]} D( v^{(n+1)}_x(z) , v^{(n)}_x(z))  &= \sup_{x \in[0,1]} D( \tilde{\Phi}(v^{(n)})_x(z) , v^{(n)}_x(z)) \\
& \leq \left( 1 + \frac{\eta^2}{||S||} \right)^{-2} \sup_{x \in[0,1]} D( v^{(n)}_x(z) , v^{(n -1)}_x(z))
\end{align*}
and so we have:

\[ \sup_{x \in[0,1]} D( v^{(n+1)}_x(z) , v^{(n)}_x(z)) \leq \frac{\delta^2}{\eta^4} \left( 1 + \frac{\eta^2}{||\tilde{S}||} \right)^{-2n}. \]

And so for $\delta$ small enough, $(v^{(n)})_{n \in \N}$ is a Cauchy sequence for the distance \newline
$d_{\mathbb{H}^+_{\eta}}(u,v) = \sup_{x \in [0,1],z \in \mathbb{H}^+_{\eta}} \arcosh (1 + D( u_x(z),v_x(z)))$ which is converging toward $\tilde{m}$ the fixed point of $\tilde{\Phi}$ and 

\[ d_{\mathbb{H}^+_{\eta}}(m , \tilde{m}) \leq  \sum_{n=0}^{+ \infty} \arcosh \left( 1 + \frac{\delta^2}{\eta^4} \left( 1 + \frac{\eta^2}{||\tilde{S}||} \right)^{-2n} \right). \]

Therefore for every $\epsilon>0$ and $\eta > 0$, there is $\delta > 0$  small enough such that $\sup_{x,y} |\sigma^2(x,y) - \tilde{\sigma}^2(x,y)| \leq \delta$ implies $ d_{\mathbb{H}_{\eta}^+}(m , \tilde{m}) \leq \epsilon$. 
Since a base of neighbourhood of $\mu_{\sigma}$ for the vague topology is given by:

\[ \mathcal{V}_{\eta} := \{ \lambda \in \mathcal{P}(\R) : \forall z \in \mathbb{H}^+_{\eta}, | G_{\mu_{\sigma}}(z) - G_{\lambda}(z)| \leq \eta \} \]
and because the vague topology and the weak topology are equal on the set $\mathcal{P}(\R)$ of probability measure on $\R$, we have our result since $G_{\mu_{\sigma}} = \int_0^1 m_x dx$ and $G_{\mu_{\tilde{\sigma}}} = \int_0^1 \tilde{m}_x dx$

\end{proof}

\begin{cor}
In the case of a continuous variance profile $\sigma$, $\mun$ converges in probability towards $\mu_{\sigma}$. 
\end{cor}

\begin{proof}
This is a consequence from Theorem \ref{Girko} and Theorem \ref{Stab1} by noticing \newline that $\lim_{N \to \infty} \sup_{x,y \in [0,1]} | (\sigma_N)^2(x,y) - \sigma^2(x,y)| = 0$.
\end{proof}

We will also need a similar result for the piecewise constant case. 
\begin{theo}\label{Stab2}
Let $s =(s_{i,j})_{i,j=1}^n \in \mathcal{S}_n(\R^+)$ and $\vec{\alpha}, \vec{\beta} \in (\R^{+,*})^n$ be two vectors of positive coordinates summing to one and let $\gamma_i = \sum_{j=0}^i \alpha_j$ and $\tilde{\gamma}_i = \sum_{j=0}^i \beta_j$. Let $\sigma$ and $\tilde{\sigma}$ and the two piecewise constant variance profiles associated respectively with the couples $(s, \vec{\alpha})$ and $(s, \vec{\beta})$ and $v$ and $\tilde{v}$ the solutions respectively of $K_{\sigma}$ and $K_{\tilde{\sigma}}$. For $i\in \llbracket 1,n \rrbracket$ let also $m_i$ and $\tilde{m}_i$ be the holomorphic functions given by $v_x = \sum_{i=1}^n \mathds{1}_{ \gamma_{i-1} \leq x < \gamma_i}  m_i$ and $\tilde{v}_x = \sum_{i=1}^n \mathds{1}_{ \tilde{\gamma}_{i-1} \leq x < \tilde{\gamma}_i}  \tilde{m}_i$. Then for every $\eta >0$ there is $\epsilon > 0$ such that if $\sup_{i} |\alpha_i - \beta_i| \leq \epsilon$, then for all $z \in \mathbb{H}^+_{\eta}$ , we have $\sup_{i} | m_i(z) - \tilde{m}_i(z)| \leq \eta$. 
\end{theo}
Note that if we impose $s_{i,j} >0$ for all $i,j$, this result is a particular case of Proposition 10.1 \cite{AltErKr}. However, since we would like for $s_{i,j}$ to be potentially $0$, we present the following proof which while not at the same level of generality and quantitative bounds will be sufficient enough. 
\begin{proof}
We use the same notations as in the previous proof. 
Since $v$ is the solution of $K_{\sigma}$, the $m_i$ satisfy the following system:
\[ m_i = \frac{1}{ z- \sum_{j=1}^n \alpha_j s^2_{i,j} m_j  }, \]
For the $\tilde{m}_i$, we have: 
\[ \tilde{m}_i = \frac{1}{ z - \sum_{j=1}^n \beta_j s^2_{i,j} \tilde{m}_j  } . \]
 For u a function from $\mathbb{H}_{\eta}^+$ to $(\mathbb{H}^-)^n$, we let $\Phi(u)$ and $\tilde{\Phi}(u)$ be two functions from $\mathbb{H}_{\eta}^+$ to $(\mathbb{H}^-)^n$ defined for $z \in \mathbb{H}_{\eta}^+$ and $i=1,\dots,n$ by: 

\[ (\Phi(u)(z))_i = \frac{1}{ z - (Su)_i(z) } \]
and 
\[ (\tilde{\Phi}(u)(z))_i = \frac{1}{z - (\tilde{S}u )_i(z) }\]
where $S$ and $\tilde{S}$ are the linear applications defined by

\[ \forall i =1,...,n ,(Su)_i = \sum_{i=1}^n \alpha_i s^2_{i,j} u_j \text{ and } (\tilde{S}u)_i = \sum_{i=1}^n \beta_i s^2_{i,j} u_j\].

As in the previous proof, if $\mathcal{B}_{\eta} := \{ u : \mathbb{H}^+_{\eta} \to (\mathbb{H}^-)^n : \sup_{z \in \mathbb{H}^+_{\eta}} \Im u(z) \leq \frac{- \eta^3}{ (2 + \min\{||S||,||\tilde{S}||\})^2},  \sup_{z \in \mathbb{H}^+_{\eta}} || u(z) || \leq \eta^{-1}\}$, $\Phi$ and $\tilde{\Phi}$ maps $\mathcal{B}_{\eta}$ onto itself for $\eta$ small enough. For $u \in \mathcal{B}_{\eta}$, we have as before if $\delta \geq (\sup_{i,j} s^2_{i,j})(\sup_{k} | \alpha_k - \beta_k| )$, for all $i$: 

\begin{align*} 
D( \Phi(u)_i(z), \tilde{\Phi}(u)_i(z)) & \leq \frac{ |(S - \tilde{S})(u)_i(z)|}{\Im z^2} \\
&\leq \frac{\delta^2}{\eta^4}
\end{align*}
Then, using the same reasoning as in the previous case, we have that for every $\eta> 0$, there is $\delta' > 0$ such that if $\sup_{i} | \alpha_i - \beta_i| \leq \delta'$ then $ \sup_{z \in \mathbb{H}^+_{\eta}} \sup_{i} | m_i(z) - \tilde{m}_i(z)| \leq \eta $.

\end{proof}

\begin{rem}\label{Moment}
	A more elementary proof of the convergence of the measure can also be obtained since we have bounds on the moments of the entries in our case via a classical moments methods. 
	Let us consider for any $k \in \N$, we consider  $W_{k}$ the set of  words $ w = (w_1,...,w_{2k +1})$ on $\llbracket 1, k+1 \rrbracket$ such that $\{ w_1,...,w_{2k +1} \} = \llbracket 1, k+1 \rrbracket$, $w_1 = w_{2k+1}$ and such that for any $i \in \llbracket 1, k+ 1 \rrbracket$, there is exactly one $j \in \llbracket 1, k+ 1 \rrbracket \setminus \{i \}$ such that $\{ w_i, w_{i+1} \} = \{ w_j, w_{j+1}\}$. For such words $w$, we define $E_{w} := \{ \{w_i,w_{i+1} \} : i \in \llbracket 1, 2k +1 \rrbracket \}$. On this set, we can define a relation of equivalence $\sim$ by letting $w \sim w'$ if there is a permutation $f$ of $\llbracket 1 , k+1 \rrbracket$ such that $f(w_i) = w'_i$ for every $i$. We can then define $\mathcal{W}_k$ an arbitrary set of representative of $W_k$ for the equivalence relation $\sim$. Then using classical arguments for the computation of moments of $\mun$, one can see, using that the $k$-th moment of the entries of $X_N^{(\beta)}$ is bounded uniformly in $N$, that we have that for $k \geq 0$: 
	
	\[ \E[\mun(x^{k})] = \frac{1}{N} \E[ Tr( X_N^{(\beta)})^k ] = c^N _k + O(N^{-1/2}) \]
	and 
	\[ Var(\mun(x^{k})) = O(N^{-1/2}) \]
	where $c^N_k = 0$ for $k$ odd and 
	\[ c^N_{2k} = N^{- k -1} \sum_{w \in \mathcal{W}_{k}} \sum_{i_1,...,i_{k+1} =1 \atop \text{ pairwise distinct} }^N \prod_{\{k,l\} \in E_w}  \Sigma^2_N(i_k, i_l)\]
	We redirect the reader to \cite[Section 2]{AGZ} for instance to get an overview of such methods in the case of classical Wigner matrices. 
	One can then find that in the piecewise constant case : 
	\[ \lim_{N \to \infty}c^N_{2k} =\sum_{w \in \mathcal{W}_{k}} \sum_{i_1,...,i_{k+1} =1 }^n \prod_{j=1}^{k+1} \alpha_{i_j}\prod_{\{k,l\} \in E_w}  \sigma^2_{i_k, i_l} \]
	and in the continuous case:
	\[ \lim_{N \to \infty}c^N_{2k} =\sum_{w \in \mathcal{W}_{k}}  \int_{[0,1]^{k+1}}\prod_{\{k,l\} \in E_w}  \sigma^2(x_k,x_{k+1}) dx_1...dx_l. \]
	So if we denote $c_k = \lim_{N \to \infty}c_k^N$, we have that in probability $\mun$ converges toward a measure $\mu_{\sigma}$ whose moments are the $c_k$. 
	\end{rem}

In order to apply the full results of \cite{Aji17}, we will need the positivity assumption \ref{Pos} for the piecewise constant variance profile. We then have the following convergence result: 

\begin{theo} \label{Convergence}
If we are in the piecewise constant case with Assumption \ref{Pos} satisfied, 
if we let $l_N$ and $r_N$ be respectively the left and right edge of the support of $\mun$, that is respectively the smallest eigenvalue and largest eigenvalue of $X_N^{(\beta)}$ and $l_{\sigma}$ and $r_{\sigma}$ the left and right edges of the support of $\mu_{\sigma}$, we have for every $\delta >0$, $D >0$,
\[ \Pp[ r_N \geq  r_{\sigma} + \delta \text{ or } l_N \leq l_{\sigma}- \delta] \leq N^{-D} \]
for $N$ large enough. 
 
\end{theo}
\begin{proof}
This is in fact an application of corollary 1.10 from \cite{Aji17} which states that the extreme eigenvalues cannot leave the neighborhood of the support of $\mu_{\sigma_N}$ where $\sigma_N$ is the same as in Theorem \ref{Girko}. We need only to check the hypothesis (A) to (D). Up to multiplication by a scalar, our matrix model satisfies the boundedness condition (A) and the Assumption \ref{Pos} gives us the positivity hypothesis (B). The sharp-sub Gaussian hypothesis gives (D). For the boundedness condition on the Stieltjes transform (C) we can use Theorem 6.1 from \cite{Aja15}. Our kernel operator $S$ satisfy assumptions A1,A2 and B1. Particularly we can use remark 6.2 and 6.3 to bound $m$ respectively away and near $0$. Then, we need only to prove that $r_{\sigma_N}$ and $l_{\sigma_N}$ converges toward $r_{\sigma}$ and $l_{\sigma}$. This can be done for instance by looking at the expression on the moments of $\mu_{\sigma_N}$ and $\mu_{\sigma}$  given in Remark \ref{Moment}. $\sigma_N$ and $\sigma$ are both piecewise constant functions with $\sigma$ being associated with the parameters $(\sigma_{i,j})_{i,j \in \llbracket 1, n \rrbracket}$ and $(\alpha_i)_{i \in \llbracket 1, n \rrbracket }$ and $\sigma_N$ being associated with the parameters $(\sigma_{i,j})_{i,j \in \llbracket 1, n \rrbracket}$ and $(\alpha_i^N)_{i \in \llbracket 1, n \rrbracket }$ where the $\alpha_{i}^N$ are defined by 

\[ \alpha_i^N := \frac{ | I_i^{(N)}|}{N} \]

where we remind that the $I_i^{(N)}$ are defined in subsection \ref{VarProf}. Since $\lim_{N\to \infty} \alpha_i^N = \alpha_i$ for any $\epsilon > 0$ we have that for $N$ large enough: 

\[ \forall i \in \llbracket 1,n \rrbracket, \  (1 - \epsilon) \alpha^N_i \leq \alpha_i \leq (1+ \epsilon) \alpha^N_i \]

Then using the formula of Remark \ref{Moment} which gives the moments of $\mu_{\sigma_N}$ and $\mu_{\sigma}$ in terms of the $\sigma_{i,j}$ and the $\alpha_i$ and $\alpha_i^N$ we see that for every $\epsilon > 0$, for $N$ large enough, we have that for every $k \in \N$, $(1- \epsilon)^{k+1} \mu_{\sigma_N}(x^{2k} )\leq \mu_{\sigma}(x^{2k}) \leq (1+ \epsilon)^{k+1} \mu_{\sigma_N}(x^{2k} )$. We conclude using that since the $\mu_{\sigma}$ and $\mu_{\sigma_N}$ are symmetric, $r_{\sigma_N} = - l_{\sigma_N} = \lim_{k \to \infty}  \mu_{\sigma_N}(x^{2k})^{1/2k}$
and $r_{\sigma} = - l_{\sigma} = \lim_{k \to \infty}  \mu_{\sigma}(x^{2k})^{1/2k}$. 
\end{proof}


\section{Appendix: Proof of Lemma \ref{convmunBM}}\label{proof}

This section is devoted to the proof of Lemma \ref{convmunBM}. For this, we will use a concentration results respectively from \cite{GZ} and Theorem \ref{Girko}

\begin{theo}\label{concentrationmu} By    \cite[(Corollary 1.4 a)] {GZ} (for the compact case) and   \cite[Corollary 1.4 b)] {GZ} (for the logarithmic Sobolev case), we have for  $\beta=1,2$,  and for  $N$ large enough
$$ \limsup_{N\ra\infty}\frac{1}{N^{7/6}}\log  \Pp[ d( \mun, \E[ \mun ]) > N^{-1/6}]  <0$$
where $d$ is the Wasserstein distance distance defined for two measure $\mu,\nu \in \mathcal{P}(\R)$ by 
\[ d( \mu, \nu ) =\sup_{f \in \mathcal{F}_{\text{Lip}} }\Big| \int fd \mu - \int fd \nu \Big| \]
with $\mathcal{F}_{\text{Lip}}$ being the set of bounded Lipschitz real function  $f$ of $\R$, such that $||f||_{\infty} + ||f||_{\text{Lip}} \leq 1$, where 
	\[ ||f||_{\text{Lip}} = \sup_{x,y \in \R \atop x \neq y} \frac{ | f(x) - f(y)|}{|x -y|} \]
\end{theo}

Therefore we only need to show that 
\[ \lim_{N \to \infty} d(\E[ \mun], \mu_{\sigma}) =0 \]

Using Theorem \ref{Stab2}, we have that $\lim_{N \to + \infty}d(\mu_{\sigma_N}, \mu_{\sigma}) = 0 $ and therefore, using Theorem \ref{Girko} we have in probability $\lim_{N \to + \infty}d(\mun, \mu_{\sigma}) = 0$. From this, we can deduce that $\lim_{N \to \infty} d( \E[\mun], \mu_{\sigma}) = 0$. Indeed, if $f$ is a continuous function bounded in absolute norm by $1$, We have in probability that $\lim_{N \to \infty}\mun(f) = \mu_{\sigma}(f)$ and so since $|\mun(f)| \leq 1$, we have $\lim_{N \to \infty}\E[\mun](f) = \mu_{\sigma}(f)$. 

%
%
%
%
%
%
%
%

\bibliographystyle{plain}
\bibliography{biblioRad}

\end{document}